\documentclass[reqno]{amsart}
\usepackage[utf8]{inputenc} 
\usepackage[T1]{fontenc}
\usepackage{ae}
\usepackage{aecompl}
\usepackage{bbold}
\usepackage{amssymb}
\usepackage[all]{xypic}
\usepackage[english]{babel}
\usepackage{extarrows}
\usepackage{xparse}
\usepackage[numbers]{natbib}
\usepackage{enumitem} 
\usepackage{url}
\usepackage{hyperref}

\title{Homotopy theory of dg sheaves}

\author{Utsav Choudhury}
\address{Department of Mathematics,
RKM Vivekananda University, India}
\email{prabrishik@gmail.com}
\author{Martin Gallauer Alves de Souza}
\address{Department of Mathematics, University of California, Los Angeles}
\email{gallauer@math.ucla.edu}
\thanks{The first author was supported by the Alexander von Humboldt Foundation, the second author by the Swiss National Science Foundation}

\date{\today{}}

\makeatletter
\let\thetitle\@title
\makeatother
\hypersetup{bookmarksdepth=1, colorlinks, citecolor=blue,
linkcolor=blue,pdfinfo={Title=\thetitle,
Author={Utsav Choudhury and Martin Gallauer Alves de Souza}}}
\newcounter{sat}
\newtheorem{thm}[sat]{Theorem}
\newtheorem{fac}[sat]{Fact}
\newtheorem{lem}[sat]{Lemma}
\newtheorem{cor}[sat]{Corollary}

\newtheorem{pro}[sat]{Proposition}
\theoremstyle{definition}
\newtheorem{dfi}[sat]{Definition}
\newtheorem{exa}[sat]{Example}
\newtheorem{rem}[sat]{Remark}
\numberwithin{sat}{section}

\makeatletter
\DeclareRobustCommand*\cal{\@fontswitch\relax\mathcal}
\makeatother

\DeclareRobustCommand{\gobblefive}[5]{}
\newcommand*{\SkipTocEntry}{\addtocontents{toc}{\gobblefive}}

\usepackage{xspace}

\providecommand{\UV}{\ensuremath{\mathbf{U}_{{\cal V}}}}
\providecommand{\Udg}{\ensuremath{\mathbf{U}_{\mathrm{dg}}}}
\providecommand{\U}{\ensuremath{\mathbf{U}}}
\providecommand{\fact}{\ensuremath{\mathrm{Fact}}}
\providecommand{\cofrep}{\ensuremath{\mathrm{CofRep}}}
\providecommand{\Mod}[1]{\ensuremath{\mathbf{Mod}(#1)}}
\providecommand{\dL}{\ensuremath{\mathrm{L}}}
\providecommand{\dR}{\ensuremath{\mathrm{R}}}
\providecommand{\Ho}{\ensuremath{\mathbf{Ho}}}

\providecommand{\god}{\ensuremath{{\cal G}}}
\DeclareMathOperator{\ch}{\mathbf{Cpl}}
\DeclareMathOperator{\psh}{\mathbf{PSh}}
\DeclareMathOperator{\sh}{\mathbf{Sh}}

\DeclareMathOperator{\dghom}{\underline{hom}_{\mathrm{dg}}}
\providecommand{\totp}{\ensuremath{\mathrm{Tot}^{\scriptscriptstyle\prod}}}
\providecommand{\tots}{\ensuremath{\mathrm{Tot}^{\oplus}}}
\newcommand{\colim}{\operatornamewithlimits{colim}}
\newcommand{\id}{\ensuremath{\mathrm{id}}}
\newcommand{\h}{\ensuremath{\mathrm{H}}}

\newcommand{\one}{\ensuremath{\mathbb{1}}}
\providecommand{\eg}{\mbox{e.\,g.}\xspace}
\providecommand{\ie}{\mbox{i.\,e.}\xspace}
\usepackage{extarrows}

\providecommand{\Z}{\ensuremath{\mathbb{Z}}}
\providecommand{\N}{\ensuremath{\mathbb{N}}}
\providecommand{\simplex}{\ensuremath{\varDelta}}
\providecommand{\simp}{\ensuremath{\mathbf{\Delta}}}
\providecommand{\totp}{\ensuremath{\mathrm{Tot}^{\scriptscriptstyle\prod}}}
\providecommand{\tots}{\ensuremath{\mathrm{Tot}^{\oplus}}}
\newcommand{\sset}{\ensuremath{\simp^{\mathrm{op}}\set}}
\newcommand{\set}{\ensuremath{\mathrm{Set}}}
\DeclareMathOperator{\hmap}{\dR\mathrm{map}}

\begin{document}

\begin{abstract}
  In this note we study the local projective model structure on
  presheaves of complexes on a site, \ie{} we describe its classes of
  cofibrations, fibrations and weak equivalences. In particular, we
  prove that the fibrant objects are those satisfying descent with
  respect to all hypercovers. We also describe cofibrant and fibrant
  replacement functors with pleasant properties.
\end{abstract}

\maketitle
\tableofcontents{}
\section{Introduction}

An important object in different fields of mathematics is the derived
category of sheaves on some site. However, it is well-known that many
constructions and proofs in this setting cannot be performed on the
derived level but require recourse to a \emph{model}. Our goal in this
note is to describe in detail one specific homotopy-theoretic model
for the unbounded derived category of sheaves on an arbitrary
site. Although the model is well-known, there were several facts about
it that we needed in our \cite{choudhury-gallauer:iso-galois} but were
not able to find in the literature, which is why we decided to write
them up. To be useful in other contexts as well, we place ourselves in
a more general setting, in particular we try to make as few
assumptions as possible regarding the site.

Let us quickly give the definition of the model associated to a site
$({\cal C},\tau)$. Start with the category of presheaves of unbounded
complexes on ${\cal C}$ and declare weak equivalences and fibrations
to be objectwise quasi-isomorphisms and epimorphisms,
respectively. This yields the \emph{projective model structure}. The
\emph{$\tau$-local model structure} arises from it by a left Bousfield
localization with respect to $\tau$-local weak equivalences, \ie{}
morphisms inducing isomorphisms on all homology $\tau$-sheaves. The
resulting model category is our model for the derived category of
$\tau$-sheaves.

In~§\ref{sec:uni-dg} we recall the basic properties of the model
category and describe the cofibrations. As an application we construct
in~§\ref{sec:cof} an explicit cofibrant replacement functor which
resolves any presheaf of complexes by representables. The main theorem
of~§\ref{sec:local} states that the $\tau$-fibrant objects are
precisely those presheaves satisfying descent with respect to
$\tau$-hypercovers. The analogous statement for simplicial presheaves
is well-known, and our strategy is to reduce to this case via the
Dold-Kan correspondence. We use the same strategy to prove a
generalization of the Verdier hypercover theorem, expressing the
hypercohomology of a complex of sheaves in terms of hypercovers. We
also describe some modifications to our model and deduce some useful
consequences from the main theorem. In the final section~\ref{sec:fib}
we prove that the Godement resolution defines a fibrant replacement
functor for our model.

We would like to remark that the model described in this note is
not quite arbitrary but has a very satisfying universal property. To
describe it, recall the easy fact from category theory that for a
small category ${\cal C}$, the category of presheaves on ${\cal C}$ is
its \emph{universal (or free) cocompletion}. This means that any
functor from ${\cal C}$ into a cocomplete category factors via a
cocontinuous functor through the Yoneda embedding ${\cal
  C}\to\psh({\cal C})$ in an essentially unique way. This basic idea
finds repercussions in the following two results:
\begin{itemize}
\item For a small dg category ${\cal D}$, the category of
  dg modules $[{\cal D}^{\mathrm{op}},\ch]$ is its
  \emph{universal dg cocompletion}.
\item In~\cite{dugger-universal-hty}, Dugger proves that any functor
  from ${\cal C}$ into a model category factors via a left Quillen
  functor through the category of simplicial presheaves on ${\cal C}$
  with the projective model structure, in an essentially unique
  way. In other words, this is the \emph{universal model category}
  associated to ${\cal C}$.
\end{itemize}
Combining these two examples we naturally arrive at the following
guess: $[{\cal C}^{\mathrm{op}},\ch]$ with the projective model
structure is the \emph{universal model dg category}
associated to the small category ${\cal C}$. We couldn't resist
prepending a section (§\ref{sec:uni-psh}) in order to explain this
result (in fact, a more general version where chain complexes are
replaced by quite arbitrary enriching categories).

Such a statement invites us to conceive of ${\cal C}$ as
\emph{generating} the dg category $[{\cal
  C}^{\mathrm{op}},\ch]$, while the Bousfield localization yielding
the local model structure plays the role of imposing
\emph{relations}. Namely, the localization stipulates that any object
in ${\cal C}$ may be homotopically decomposed into the pieces of any
cover. In a very precise sense then
(cf.~Corollary~\ref{universal-local-dg-model}) our model for the
derived category of $\tau$-sheaves is the universal $\tau$-local model
dg category associated to ${\cal C}$. 

\SkipTocEntry\subsection*{Relation to other works in the literature}
As mentioned above, our motivation for this note lies
in~\cite{choudhury-gallauer:iso-galois}. Most importantly we needed
there a description of the fibrant objects in the local projective
model structure in terms of descent. At the time we were aware of such
descriptions in the "non-linear" case of simplicial presheaves due to
Dugger-Hollander-Isaksen (\cite{dugger-hollander-isaksen}), and in the
"linear" case only under finiteness conditions not satisfied in our
application (\cite{vezzani-fw-14}).  Only after the note had been
written we came across the paper~\cite{hinich:def-sh-alg} by Hinich
which establishes both the linear and the non-linear case without
restrictions (and with a proof different from ours,
cf.~Remark~\ref{existence-local-model-structure}). We believe that the
other results presented here are probably known even if they haven't
all appeared in print. The present note thus serves primarily as a
reference for~\cite{choudhury-gallauer:iso-galois} but we hope it will
be useful to other mathematicians as well.
\section{Universal enriched model categories}
\label{sec:uni-psh}

This section is very much inspired by
Dugger's~\cite{dugger-universal-hty} where he proves the existence of
a universal model category associated to a small category. Our goal is
to establish an analogue of this result in the enriched setting.

``Monoidal'' is an abbreviation for ``unital monoidal''; the monoidal
structure is always denoted by $\otimes$, the unit by $\one$. Fix a
bicomplete closed symmetric monoidal category ${\cal V}$. We are first
going to recall some basics in ${\cal V}$-enriched category theory,
and for this we follow the terminology in~\cite{kelly-enriched-1982}.

\subsection{Free enriched cocompletion}
\label{sec:enriched-kan}
Let ${\cal C}$ and ${\cal M}$ be ${\cal V}$-categories and assume that
${\cal C}$ is small. Recall (\cite[§2]{kelly-enriched-1982}) that
there is a ${\cal V}$-functor category $[{\cal C},{\cal M}]$ whose
underlying category is just the category of ${\cal V}$-functors ${\cal
  C}\to {\cal M}$ together with ${\cal V}$-natural
transformations. Given such a ${\cal V}$-functor $\gamma:{\cal
  C}\to{\cal M}$ consider the ${\cal V}$-functor $\gamma_{*}:{\cal
  M}\to [{\cal C}^{\mathrm{op}},{\cal V}]$ which takes $m$ to ${\cal
  M}(\gamma(\bullet),m)$. In particular, if ${\cal C}={\cal M}$ and
$\gamma$ the identity then $\gamma_{*}$ is the Yoneda embedding
$y:{\cal C}\to[{\cal C}^{\mathrm{op}},{\cal V}]$. As in the classical
case, the Yoneda embedding provides the free cocompletion as we are
now going to explain (see~\cite[Thm.~4.51]{kelly-enriched-1982}).

Recall that a ${\cal V}$-category ${\cal M}$ is cocomplete if it has
all small indexed colimits (sometimes also called weighted
colimits). In practice, the functors ${\cal M}(\bullet,m)_{0}:{\cal
  M}^{\mathrm{op}}_{0}\to{\cal V}_{0}$ often preserve limits (for
example, if ${\cal M}$ is cotensored or if ${\cal V}$ is
conservative). 
In this case cocompleteness is equivalent to ${\cal M}$ being tensored
and the underlying category being cocomplete in the ordinary
sense. The first condition means that there exists a ${\cal
  V}$-bifunctor (called the tensor)
 \begin{equation*}
   \bullet\odot\bullet:{\cal V}\otimes{\cal M}\to{\cal M}
 \end{equation*}
 together with, for each $v\in{\cal V}$ and each
 $m\in{\cal M}$, ${\cal V}$-natural isomorphisms
 \begin{equation*}
   {\cal M}(v\odot m, \bullet)\cong {\cal V}(v,{\cal M}(m,\bullet)).
 \end{equation*}
 Accordingly, a ${\cal V}$-functor is cocontinuous if and only if it
 commutes with tensors and the underlying functor is
 cocontinuous. 
 Dually one defines complete ${\cal V}$-categories and continuous
 ${\cal V}$-functors.

 An example of a cocomplete ${\cal V}$-category is $[{\cal
   C}^{\mathrm{op}},{\cal V}]$ for a small ${\cal V}$-category ${\cal
   C}$. From now on, we denote it by $\UV{\cal
   C}$. The tensor of $v\in{\cal V}$ and $f\in\UV{{\cal C}}$ is given by
 \begin{equation*}
   v\odot f= v_{\mathrm{cst}}\otimes f,
 \end{equation*}
 where $v_{\mathrm{cst}}$ denotes the constant presheaf with value
 $v$, and $\otimes$ denotes the objectwise tensor product in ${\cal
   V}$.
 \begin{fac}\label{enriched-kan} Let $\gamma:{\cal
     C}\to{\cal M}$ be a ${\cal V}$-functor and assume that ${\cal C}$
   is small and ${\cal M}$ is cocomplete.
   \begin{enumerate}
   \item There is a ${\cal V}$-adjunction
     \begin{equation*}
       (\gamma^{*},\gamma_{*}):\UV{\cal C}\to{\cal M},
     \end{equation*}
     where $\gamma^{*}(f)$ is given by the tensor product of $f$ and
     $\gamma$, $f\odot_{{\cal C}}\gamma$.
   \item The association $\gamma\mapsto\gamma^{*}$ induces an
     equivalence of ${\cal V}$-categories
     \begin{equation*}
       [{\cal C},{\cal M}]\simeq [\UV{\cal C},{\cal M}]_{\mathrm{coc}}
     \end{equation*}
     where $(\bullet)_{\mathrm{coc}}$ picks out the cocontinuous
     ${\cal V}$-functors.
   \item There is a canonical isomorphism $\gamma^{*}y\cong \gamma$.
   \end{enumerate}
 The ${\cal V}$-functor $\gamma^{*}$ is called the \emph{left ${\cal
     V}$-Kan extension} of $\gamma$ along the Yoneda embedding.
 \end{fac}
 Here, the tensor product of the two ${\cal V}$-functors $f$ and
 $\gamma$ is the coend $\int^{c\in{\cal C}}f(c)\odot
 \gamma(c)$. Notice that part of the statement is the existence of
 $[\UV{\cal C},{\cal M}]_{\mathrm{coc}}$ as a ${\cal V}$-category
 (this is not clear since $\UV{\cal C}$ is not necessarily small).

 If $\beta:{\cal D}\to{\cal C}$ is a ${\cal V}$-functor between small
 ${\cal V}$-categories, we denote $(y\beta)^{*}$ by $\beta^{*}$ if no
 confusion is likely to arise. With this abuse of notation, there is a
 canonical isomorphism
 $(\gamma\beta)^{*}\cong\gamma^{*}\beta^{*}$. Similarly, if
 $\delta:{\cal M}\to{\cal N}$ is a cocontinous functor into another
 cocomplete ${\cal V}$-category ${\cal N}$, then
 $(\delta\gamma)^{*}\cong \delta\gamma^{*}$.

 Assume now that ${\cal C}$ is a (symmetric) monoidal ${\cal
   V}$-category (this is the canonical translation of a (symmetric)
 monoidal structure to the enriched context; or
 see~\cite[p.~2f]{day:closed-cat-functors}). $\UV{\cal C}$ inherits a
 (symmetric) monoidal structure called the (Day) convolution product
 (\cite[Thm.~3.3 and 4.1]{day:closed-cat-functors}). Explicitly, the
 monoidal product of two presheaves $f$ and $g$ is given by
 \begin{equation*}
   f\otimes g=\int^{c,c'} f(c)\otimes g(c')\otimes{\cal
     C}(\bullet,c\otimes c'),
 \end{equation*}
 and the unit by $y(\one)={\cal C}(\bullet,\one)$. It is clear that
 the Yoneda embedding $y:{\cal C}\to\UV{\cal C}$ is (symmetric)
 monoidal.
\begin{lem}\label{enriched-kan-mon}
  In the setting of Fact~\ref{enriched-kan}, assume in addition that
  $\gamma$ is (lax) (symmetric) monoidal, and that the monoidal
  product in ${\cal M}$ commutes with indexed colimits. Then:
  \begin{enumerate}
  \item $\gamma^{*}$ is (lax) (symmetric) monoidal.
  \item The canonical isomorphism $\gamma^{*}y\cong \gamma$ is
    monoidal.
  \item The association $\gamma\mapsto\gamma^{*}$ induces an
    equivalence of ordinary categories
    \begin{equation*}
      {\cal V}\mathrm{-Fun}_{\otimes}({\cal C},{\cal M})\simeq {\cal V}\mathrm{-Fun}_{\mathrm{coc},\otimes}(\UV{\cal C},{\cal M})
    \end{equation*}
    of (lax) (symmetric) monoidal ${\cal V}$-functors.
  \end{enumerate}
\end{lem}
\begin{proof}
  Let $f,g\in\UV{\cal C}$. The (lax) monoidal structure on $\gamma^{*}$
  is defined as follows:
    \begin{align*}
      \left( \int f\odot \gamma \right)\otimes \left( \int g\odot \gamma
    \right)      &\cong \int^{c,d}(f(c)\otimes g(d))\odot (\gamma(c)\otimes
      \gamma(d))\\
      &\to \int^{c,d}(f(c)\otimes g(d))\odot (\gamma(c\otimes d))\\
      &\cong \int^{e} \left( \int^{c,d}f(c)\otimes
        g(d)\otimes {\cal C}(e,c\otimes d)
      \right)\odot \gamma(e)\\
      &\cong \int (f\otimes g) \odot \gamma
    \end{align*}
  and
  \begin{equation*}
    \one\to \gamma(\one)\cong \gamma^{*}(\one).
  \end{equation*}
We leave the details to the reader.
\end{proof}
In this sense, if ${\cal C}$ is (symmetric) monoidal then $\UV{\cal
  C}$ is the free \emph{(symmetric) monoidal} ${\cal
  V}$-cocompletion. Notice also that the ``pseudo-functoriality''
mentioned above, to wit $(\gamma\beta)^{*}\cong\gamma^{*}\beta^{*}$
and $(\delta\gamma)^{*}\cong\delta\gamma^{*}$, is compatible with
monoidal structures.
\subsection{Enriched model categories}
\label{sec:v-modcat}
We now discuss the interplay between basic enriched category theory as
above and Quillen model structures. From now on we assume that the
underlying category ${\cal V}_{0}$ is a symmetric monoidal model
category in the sense of~\cite[Def.~4.2.6]{hovey-modelcategories}. We
also assume that this model structure is cofibrantly generated.

Fix a small ordinary category ${\cal C}$ and set ${\cal C}[{\cal V}]$
to be the associated free ${\cal V}$-category. It has the same objects
as ${\cal C}$ and the ${\cal V}$-structure is given by
\begin{equation*}
  {\cal C}[{\cal V}](c,c')=\coprod_{{\cal C}(c,c')}\one
\end{equation*}
with an obvious composition. By definition, giving a ${\cal
  V}$-functor ${\cal C}[{\cal V}]\to{\cal M}$ into a ${\cal V}$-model
category ${\cal M}$ is the same as giving an (ordinary) functor ${\cal
  C}\to{\cal M}_{0}$. In the sequel, we will often write abusively
${\cal C}\to{\cal M}$, sometimes thinking of the datum as a ${\cal
  V}$-functor, sometimes as an ordinary functor. We are positive that
this will not lead to any confusion.

Thus the general small ${\cal V}$-category ${\cal C}$
in~§\ref{sec:enriched-kan} will now always be of this special form. We
impose this restriction because it simplifies most of the statements
and proofs drastically, and because it is all we will need later on.

Since the underlying category of $\UV{\cal C}:=\UV{\cal C}[{\cal V}]$
is just the category of presheaves on ${\cal C}$ with values in ${\cal
  V}$, the following result is well-known.
\begin{fac}\label{functorcat-model}\mbox{}
  \begin{enumerate}
  \item $(\UV{\cal C})_{0}$ admits a cofibrantly generated model
    structure with weak equivalences and fibrations defined objectwise.
   \item If ${\cal V}_{0}$ is left (resp. right) proper then so
    is $(\UV{\cal C})_{0}$.
  \item If ${\cal V}_{0}$ is combinatorial (resp. tractable, cellular)
    then so is $(\UV{\cal C})_{0}$.
  \item If ${\cal V}_{0}$ is stable then so is $(\UV{\cal C})_{0}$.
  \end{enumerate}
\end{fac}
This is called the \emph{projective model structure}. If not mentioned
otherwise, we will consider $\UV{\cal C}$ as endowed with the
projective model structure from now on.
\begin{proof}
  See~\cite[Thm.~11.6.1]{hirschhorn:model-cat-loc} for the first
  statement, \cite[Thm.~13.1.14]{hirschhorn:model-cat-loc} for the
  second, \cite[Pro.~12.1.5]{hirschhorn:model-cat-loc} and
  \cite[Thm.~2.14]{barwick-modcat-bousfield} for the third. The last
  statement is obvious.
\end{proof}

\begin{dfi}
  Let ${\cal M}$ and ${\cal N}$ be ${\cal V}$-categories with model
  structures on their underlying categories. A ${\cal V}$-adjunction
  $(\delta,\varepsilon):{\cal M}\to{\cal N}$ is called a \emph{Quillen
    ${\cal V}$-adjunction} if the underlying adjunction
  $(\delta_{0},\varepsilon_{0}):{\cal M}_{0}\to{\cal N}_{0}$ is a
  Quillen adjunction. In that case $\delta$ is called a \emph{left},
  $\varepsilon$ a \emph{right Quillen ${\cal V}$-functor}.
\end{dfi}
We now come back to the situation of Fact~\ref{enriched-kan}. The
question we should like to answer is: When is
$(\gamma^{*},\gamma_{*})$ a Quillen ${\cal V}$-adjunction?
\begin{lem}\label{enriched-kan-left-quillen}Assume that ${\cal M}_{0}$ is
  endowed with a model structure. The following conditions are equivalent:
  \begin{enumerate}
  \item $(\gamma^{*},\gamma_{*})$ is a Quillen ${\cal V}$-adjunction.
  \item For each $c\in\mathcal{C}$,
    ${\cal M}_{0}(\gamma(c),\bullet)$ is a right Quillen functor.
  \item For each $c\in\mathcal{C}$, $\bullet\odot \gamma(c)$ is a left
    Quillen functor.
  \end{enumerate}
\end{lem}
\begin{proof}
  The equivalence between the last two conditions is clear. The
  equivalence between the first two conditions follows from the
  description of $\gamma_{*}$ given above and the fact that we imposed
  the projective model structure on $(\UV{\cal C})_{0}$.
\end{proof}
In particular, these equivalent conditions are satisfied if the image
of $\gamma$ consists of cofibrant objects, and the tensor on ${\cal
  M}$ is a ``Quillen ${\cal V}$-adjunction of two variables'', \ie{} a
${\cal V}$-adjunction of two variables such that the underlying data
form a Quillen adjunction of two variables in the sense
of~\cite[Def.~4.2.1]{hovey-modelcategories}.
\begin{dfi}
  A \emph{model ${\cal V}$-category} is a bicomplete ${\cal
    V}$-category ${\cal M}$ together with a model structure on ${\cal
    M}_{0}$ such that
  \begin{itemize}
  \item the tensor is a Quillen ${\cal V}$-adjunction of two variables;
  \item for any cofibrant object $m\in{\cal M}$, $\one_{c}\odot m\to
    \one\odot m$ is a weak equivalence, for a cofibrant replacement
    $\one_{c}\to\one$.
  \end{itemize}

  A \emph{(symmetric) monoidal model ${\cal V}$-category} is a model
  ${\cal V}$-category ${\cal M}$ together with a Quillen ${\cal
    V}$-adjunction of two variables $\otimes:{\cal M}\otimes{\cal
    M}\to{\cal M}$ with a unit, and associativity (and symmetry)
  constraints satisfying the usual axioms.
\end{dfi}
These are equivalent to the definitions in~\cite[Def.~4.2.18,
4.2.20]{hovey-modelcategories}. Also, it is a straight-forward
generalization of the notion of a simplicial model
category.
\begin{exa}\mbox{}
  \begin{enumerate}
  \item If ${\cal V}$ is the category of simplicial sets with the
    standard model structure then we recover the notion of a
    simplicial model category.
  \item Our main example will be obtained by taking ${\cal V}$ to be
    the category of (unbounded) chain complexes of $\Lambda$-modules,
    $\Lambda$ a (commutative unital) ring, with the projective model
    structure and the usual tensor product. A model ${\cal
      V}$-category will be called a model dg category. See
    §\ref{sec:uni-dg}.
  \end{enumerate}
\end{exa}

\begin{fac}\label{uni-model-v-cat}
  $\UV{\cal C}$ is a model ${\cal V}$-category. Moreover, if ${\cal
    C}$ is (symmetric) monoidal and the unit in ${\cal V}$ cofibrant,
  then $\UV{\cal C}$ is a (symmetric) monoidal model ${\cal
    V}$-category for the Day convolution product.
\end{fac}
\begin{proof}
  The first statement is straightforward to check. The second
  is~\cite[Pro.~2.2.15]{Isaacson:cubical-hty-mon-mod-cat}.
\end{proof}

Notice that if ${\cal C}$ is cartesian monoidal then the Day
convolution product coincides with the objectwise monoidal product on
$\UV{\cal C}$.
\subsection{Statement and proof}
\label{sec:result}
Our goal is to establish $y:{\cal C}\to\UV{\cal C}$ (really, ${\cal
  C}[{\cal V}]\to \UV{\cal C}$) as the universal functor into a model
${\cal V}$-category. But first, we need to make precise what we mean
by the universality in the statement. For this fix a model ${\cal
  V}$-category ${\cal M}$ and a functor $\gamma:{\cal C}\to{\cal
  M}$. Define a \emph{factorization of $\gamma$ through $y$} to be a
pair $(L,\eta)$ where $L:\UV{\cal C}\to{\cal M}$ is a left Quillen
${\cal V}$-functor, and $\eta:Ly\to \gamma$ a natural transformation
which is objectwise a weak equivalence. A morphism of such
factorizations $(L,\eta)\to (L',\eta')$ is a natural transformation
$L\to L'$ compatible with $\eta$ and $\eta'$. This clearly defines a
category $\fact(\gamma,y)$.

\begin{pro}\label{pro:universality}
  Assume that the unit in ${\cal V}$ is cofibrant. For any $\gamma$,
  the category $\fact(\gamma,y)$ is contractible.
\end{pro}
Notice that in a homotopical context it is unreasonable to expect the
category of choices to be a groupoid (``uniqueness up to unique
isomorphism'') and contractibility is usually the right thing to ask
of this category. 

Let $\cofrep(\gamma)$ be the category of cofibrant replacements of
$\gamma$. Its objects are functors $\gamma':{\cal C}\to{\cal M}$
together with a natural transformation $\gamma'\to\gamma$ which is
objectwise a weak equivalence and such that the image of $\gamma'$ is
cofibrant. The morphisms are the obvious ones.

\begin{lem}
  Assume that the unit in ${\cal V}$ is cofibrant. There is a
  canonical equivalence of categories
  $\fact(\gamma,y)\simeq\cofrep(\gamma)$.
\end{lem}
\begin{proof}
  We give functors in both directions. That these are quasi-inverses
  to each other will then be seen to follow from the ${\cal
    V}$-equivalence of categories in Fact~\ref{enriched-kan}.
  \begin{itemize}
  \item Given $\gamma'\to\gamma$ on the right hand side, define
    $L=(\gamma')^{*}$ and choose the natural transformation
    $(\gamma')^{*}y\cong \gamma'\to\gamma$. Functoriality follows from
    the functoriality statement in Fact~\ref{enriched-kan}.
  \item Given $(L,Ly\to\gamma)$ on the left hand side, $Ly\to\gamma$
    defines a cofibrant replacement since $L$ is a left Quillen ${\cal
      V}$-functor and the image of $y$ is
    cofibrant. Functoriality is obvious.\qedhere{}
  \end{itemize}
\end{proof}
\begin{proof}[Proof of Proposition~\ref{pro:universality}]
  By the previous lemma, we need to show contractibility of
  $\cofrep(\gamma)$. Fix a cofibrant replacement functor $F$ for the
  model structure on ${\cal M}_{0}$. Composing with $\gamma$ we obtain
  an object $(F\gamma,F\gamma\to\gamma)$ of $\cofrep(\gamma)$. Given
  any other object $(\gamma',\gamma'\to\gamma)$, functoriality of $F$
  yields a commutative square
  \begin{equation*}
    \xymatrix{F\gamma'\ar[r]\ar[d]&F\gamma\ar[d]\\
      \gamma'\ar[r]&\gamma}
  \end{equation*}
  and thus a zig-zag $\gamma'\leftarrow F\gamma'\to F\gamma$ in
  $\cofrep(\gamma)$. Moreover, this zig-zag is natural in $\gamma'$
  hence this construction provides a zig-zag of homotopies between the
  identity functor on $\cofrep(\gamma)$ and the constant functor
  $(F\gamma,F\gamma\to\gamma)$.
\end{proof}

For the reader's convenience we reformulate our main result.
\begin{cor}\label{enriched-model-universal}
  Let ${\cal C}$ be a small category, and ${\cal V}$ a cofibrantly
  generated symmetric monoidal model category whose unit is
  cofibrant. There exists a functor $y:{\cal C}\to \UV{\cal C}$ into a
  model ${\cal V}$-category, universal in the sense that for any solid
  diagram
  \begin{equation*}
    \xymatrix{{\cal C}\ar[rd]_{\gamma}\ar[r]^{y}&\UV{\cal
        C}\ar@{.>}[d]^{L}\\
      &{\cal M}}
  \end{equation*}
  with ${\cal M}$ a model ${\cal V}$-category, there exists a left
  Quillen ${\cal V}$-functor $L$ as indicated by the dotted arrow,
  unique up to a contractible choice, making the diagram commutative
  up to a weak equivalence $Ly\to\gamma$.
\end{cor}

\begin{rem}
  One can dualize the discussion of this section in order to obtain
  universal model ${\cal V}$-categories for \emph{right} Quillen
  ${\cal V}$-functors, as
  in~\cite[§4]{dugger-universal-hty}. Unsurprisingly, one finds that
  this universal model ${\cal V}$-category associated to ${\cal C}$ is
  given by $[{\cal C},{\cal V}]^{\mathrm{op}}$ with the opposite of
  the projective model structure. This can also be deduced from
  Corollary~\ref{enriched-model-universal} applied to ${\cal
    C}^{\mathrm{op}}$.
\end{rem}

\section{Universal model dg categories}
\label{sec:uni-dg}
We now specialize to the case of dg categories. Fix a
commutative unital ring $\Lambda$, denote by $\Mod{\Lambda}$ the
category of $\Lambda$-modules, and by $\ch(\Lambda)$ the category of
unbounded chain complexes of $\Lambda$-modules. Our conventions for
chain complexes are homological, \ie{} the differentials decrease the
indices, and the shift operator satisfies $(A[p])_{n}=A_{p+n}$. The
subobject of $n$-cycles (resp.\ $n$-boundaries) of $A$ is denoted by
$Z_{n}A$ (resp.\ $B_{n}A$). As usual, the $n$th homology is denoted by
$\h_{n}A=Z_{n}A/B_{n}A$.

$\ch(\Lambda)$ has a tensor product, defined by
\begin{equation*}
  (A\otimes B)_{n}=\oplus_{p+q=n}A_{p}\otimes B_{q}
\end{equation*}
with the Koszul sign convention for the differential. It also admits
the ``projective model structure'' for which the weak equivalences are
the quasi-isomorphisms, and the fibrations the epimorphisms (\ie{} the
degreewise surjections). In that way, $\ch(\Lambda)$ becomes a
symmetric monoidal model category. In this section we always take
${\cal V}$ to be $\ch(\Lambda)$. The universal model category
underlying a model dg category $(\Udg{\cal C})_{0}$ will
now be denoted by $\U{\cal C}$. The complex of morphisms from $K$ to
$K'$ in $\Udg{\cal C}$ is denoted by
$\dghom(K,K')\in\ch(\Lambda)$. Recall that it is given explicitly by
$\totp(\hom_{\psh({\cal C},\Lambda)}(K_{p},K'_{q}))_{p,q}$.

Our main goal in this section is to better understand the model
structure on $\U{\cal C}$ (defined in Fact~\ref{functorcat-model}). In
the last part we will also discuss a specific instance of a left
dg Kan extension used in~\cite{choudhury-gallauer:iso-galois}.

\subsection{Basic properties of the model category \texorpdfstring{$\U{\cal C}$}{UC}}

By Fact~\ref{uni-model-v-cat} we know that $\Udg{\cal C}$ is a model
dg category, and a (symmetric) monoidal model
dg category if ${\cal C}$ is (symmetric) monoidal. It
follows from Fact~\ref{functorcat-model} that the model category
$\U{\cal C}$ is about as nice as it can get.
\begin{cor}\label{dg-model-prop}
$\U{\cal C}$ is a
  \begin{enumerate}
  \item proper,
  \item stable,
  \item tractable (in particular combinatorial),
  \item cellular,
  \end{enumerate}
  model category.
\end{cor}
We will now describe explicitly sets of generating (trivial)
cofibrations.

\begin{dfi} \label{generators} Let, for any presheaf $F$, $S^nF$ be
  the complex of presheaves which has $F$ in degree $n$ and is $0$
  otherwise, and let $D^nF$ be the complex of presheaves which has $F$
  in degree $n$ and $n-1$, is $0$ otherwise, and whose nontrivial
  differential is given by the identity on $F$. There exists a
  canonical morphism $S^{n-1}F \to D^nF$. Let $I$ be the set of
  morphisms $S^{n-1} \Lambda(c) \to D^{n} \Lambda(c)$ for all $c \in
  {\cal C}$ and let $J$ be the set of maps $0 \to D^n \Lambda(c)$.
\end{dfi}

Notice that there are adjunctions
\begin{equation*}
(S^{n},Z_{n})\text{ and }  (D^n,(\bullet)_{n}):\psh({\cal C}, \Lambda) \to
  \U{\cal C}.
\end{equation*}
The same arguments as in~\cite[Pro.~2.3.4,
2.3.5]{hovey-modelcategories} then establish the following result.

\begin{fac} \label{I} A morphism in $\U{\cal C}$ is a fibration
  (resp. trivial fibration) if and only if it has the right lifting
  property with respect to $J$ (resp. $I$).
\end{fac}

We will use another set of generating cofibrations later on.
\begin{dfi}
  Given a presheaf $F$ of $\Lambda$-modules, let $\simplex^n F$
  be the complex which has $F$ in degree $n$ and $F \oplus F$ in
  degree $n-1$, and zero otherwise, and whose only non-zero
  differential is given by $\id\times(-\id) : F \to F \oplus F$. Define
  also $\partial\simplex^n F$ to be the complex which has $F
  \oplus F$ in degree $n-1$ and $0$ otherwise. Let $I'$ be the set of
  morphisms $\partial\simplex^n \Lambda(c) \to \simplex^{n}
  \Lambda(c)$ which is the identity in degree $n$, for all $n \in \Z$
  and $c \in {\cal C}$ .
\end{dfi}

\begin{lem} \label{I'} A morphism in $\U{\cal C}$ is a trivial
  fibration if and only if it has the right lifting property with
  respect to $I'$.
\end{lem}
\begin{proof}
  Morphisms in $I'$ are cofibrations by Fact~\ref{cof}. Conversely we
  will exhibit any morphism in $I$ as a retract of some morphism in
  $I'$. Thus fix $c\in{\cal C}$ and $n\in\Z$, and consider the
  following diagram:
  \begin{equation*}
    \xymatrix{\ S^n \Lambda(c) \ar[r]^{\id\times(-\id)} \ar[d] &\partial\simplex^{n+1} \Lambda(c) \ar[r]^{(\id,0)} 
      \ar[d] & S^n \Lambda(c) \ar[d] \\
      D^{n+1} \Lambda(c) \ar[r]^{r} &\simplex^{n+1} \Lambda(c)
      \ar[r]^s & D^{n+1} \Lambda(c)}
  \end{equation*}
  Here, $r$ in degree $n$ is $\id\times (-\id)$ and in degree $n+1$ is
  $\id$, while $s$ in degree $n$ is the first projection and in degree
  $n+1$ the identity. It is easy to see that the diagram commutes and
  the compositions of each row are the identity morphism.
\end{proof}

\subsection{Projective cofibrations}
Since the fibrations and weak equivalences are given explicitly in
$\U{\cal C}$ our goal is to better understand the cofibrations. They
are called projective cofibrations. The discussion runs parallel to
the description of projective cofibrations for the category of chain
complexes (\ie{} the case of ${\cal C}$ the terminal category),
in~\cite[§2.3]{hovey-modelcategories}.

\begin{lem} \label{t-f-ker} If $f : K \to K' \in \U{\cal C}$ is a
  trivial fibration then $f$ induces a surjective morphism $f :
  Z_{n}K\to Z_{n}K'$ for all $n\in\Z$.
\end{lem}
\begin{proof}
  Since $f$ is degreewise surjective, it induces a surjective morphism
  on the boundaries $B_{n}K\to B_{n}K'$. Now consider the morphism of
  exact sequences:
  \begin{equation*}
    \xymatrix{ B_{n}K \ar[r] \ar[d] & Z_{n}K \ar[r] \ar[d] &\h_nK \ar[d]\\
      B_{n}K' \ar[r] & Z_{n}K' \ar[r] & \h_nK'}
\end{equation*}
The first and last vertical arrows are surjective, hence the middle
one is too.
\end{proof}

\begin{dfi}
  A presheaf of $\Lambda$-modules $F \in \psh({\cal C}, \Lambda)$ is
  called \emph{projective} if
  \begin{equation*}
    \hom_{\psh({\cal
        C},\Lambda)}(F,\bullet) : \psh({\cal C}, \Lambda) \to
    \Mod{\Lambda}
  \end{equation*}
  is exact.
\end{dfi}

\begin{exa}
  For any $c \in {\cal C}$ the representable presheaf $\Lambda(c)$ is
  projective. Direct sums of projectives are projective.
\end{exa}

\begin{lem}\label{projective-cofibrant}
  For any projective presheaf $F \in \psh({\cal C}, \Lambda)$, the
  complex $S^{0}F$ is projective cofibrant.
\end{lem}
\begin{proof}
  We have to prove that for any trivial fibration $f:K \to K' \in
  \U{\cal C}$, the induced morphism
  \begin{equation*}
    \hom_{\U{\cal C}}(S^{0}F, K) \to \hom_{\U{\cal C}}(S^{0}F,K')
  \end{equation*}
  is surjective. But for any complex $L \in \U{\cal C}$, we have
  \begin{equation*}
    \hom_{\U{\cal C}}(S^{0}F, L) = \hom_{\Mod{\Lambda}}(F, Z_{0}L).
  \end{equation*}
  Now the result follows from Lem.~\ref{t-f-ker}.
\end{proof}

\begin{fac} \label{bnd} Let $K \in \U{\cal C}$. If $K$ is projective
  cofibrant then each $K_n$ is a projective presheaf. As a partial
  converse, if $K$ is bounded below and each $K_i$ is projective then
  $K$ is projective cofibrant.
\end{fac}
\begin{proof}
  The proof of~\cite[Lemma 2.3.6]{hovey-modelcategories} applies.
\end{proof}

\begin{fac} \label{cof} A map $f : K \to K' \in \U{\cal C}$ is a
  projective cofibration if and only if $f$ is a degreewise split
  injection and the cokernel of $f$ is projective cofibrant.
\end{fac}
\begin{proof}
  The proof of~\cite[Pro.~2.3.9]{hovey-modelcategories} applies.
\end{proof}

\begin{cor} \label{cof-un} Let $K=\varinjlim_{n \in \N} K^{(n)}\in
  \U{\cal C}$, such that $K^{(n)}$ is projective cofibrant and
  bounded below for each $n$, and such that the transition morphisms
  $K^{(n)} \to K^{(n+1)}$ are degreewise split injective. Then $K$ is
  projective cofibrant.
\end{cor}
\begin{proof}
  We use the fact that $K$ is a sequential colimit of projective
  cofibrant objects with transition morphisms which are split
  injective in each degree hence the cokernel has projective objects
  in each degree. This implies together with boundedness and the
  previous lemma that the transition morphisms are projective
  cofibrations. Hence $K$ is projective cofibrant.
\end{proof}

Independently of monoidal structures on ${\cal C}$, we can always
define an objectwise tensor product on presheaves. The following lemma
gives a necessary and sufficient condition for this product to be a
Quillen bifunctor.

\begin{lem}\label{uni-dg-mon-mod}
  $\U{\cal C}$ is a symmetric monoidal model category for the
  objectwise tensor product if and only if for any pair of objects
  $c,d\in{\cal C}$, the presheaf of $\Lambda$-modules
  $\Lambda(c)\otimes\Lambda(d)$ is projective.
\end{lem}
\begin{proof}
  Since representables are cofibrant (Fact~\ref{bnd}) the condition
  is clearly necessary. For the converse, it suffices to prove the
  pushout-product $i\square j$ a (trivial) cofibration if $i$ and $j$
  are generating cofibrations (and one of them a generating trivial
  cofibration). By Fact~\ref{I}, $i$ and $j$ are of the form $i'\odot
  \Lambda(c)$ and $j'\odot\Lambda(d)$ for cofibrations $i'$, $j'$ of
  $\ch(\Lambda)$ (one of which is acylic), $c, d\in{\cal
    C}$. $i\square j$ can then be identified with $(i'\square j')\odot
  (\Lambda(c)\otimes\Lambda(d))$. $i'\square j'$ is a (trivial)
  cofibration since $\ch(\Lambda)$ is a symmetric monoidal model
  category. If $\Lambda(c)\otimes\Lambda(d)$ is projective then
  Lemma~\ref{projective-cofibrant} together with
  Fact~\ref{uni-model-v-cat} yields what we want.
\end{proof}

\subsection{Dold-Kan correspondence}
\label{sec:dold-kan}
Fix an abelian category ${\cal A}$. We start by recalling some basic
constructions relating simplicial objects and connective chain
complexes in ${\cal A}$.

Given a simplicial object $a_{\bullet}$ in ${\cal A}$, one can
associate to it a connective chain complex (called the Moore complex,
and usually still denoted by $a_{\bullet}$) which is $a_{n}$ in degree
$n$ and whose differentials are given by
\begin{equation*}
  \sum_{i=0}^{n}(-1)^{i}d_{i}:a_{n}\to a_{n-1}.
\end{equation*}
This clearly defines a functor $\simp^{\mathrm{op}}{\cal
  A}\to\ch_{\geq 0}({\cal A})$. Since every object in
$\simp^{\mathrm{op}}{\cal A}$ is canonically split, we get a second
functor $N:\simp^{\mathrm{op}}{\cal A}\to\ch_{\geq 0}({\cal A})$,
which associates to $a_{\bullet}$ the normalized chain complex:
\begin{align*}
  N(a_{\bullet})_{n}=\bigcap_{i=0}^{n-1}\ker(d_{i}:a_{n}\to a_{n-1}),&&(-1)^{n}d_{n}:N(a_{\bullet})_{n}\to N(a_{\bullet})_{n-1}.
\end{align*}
Clearly, there is a canonical embedding $N(a_{\bullet})\subset
a_{\bullet}$ but more is true:
\begin{fac}\label{moore-normalized}\mbox{}
  \begin{enumerate}
  \item The inclusion $N(a_{\bullet})\to a_{\bullet}$ is a natural
    chain homotopy equivalence.
  \item There is a functorial splitting
    $a_{\bullet}=N(a_{\bullet})\oplus N'(a_{\bullet})$ and $N'$ is an
    acyclic functor.
  \item $N$ is an equivalence of categories with quasi-inverse $\Gamma$.
  \item For any $n\in\N$, there is a natural isomorphism
    $\pi_{n}\Gamma \cong \h_{n}$.
  \end{enumerate}
\end{fac}

In particular, we obtain a sequence of adjunctions
\begin{equation}\label{eq:dold-kan-factorization}
\xymatrix{\sset\ar@<.7ex>[r]^-{\Lambda}&\simp^{\mathrm{op}}\Mod{\Lambda}\ar@<.7ex>[r]^-{N}\ar@<.7ex>[l]&\ch_{\geq 0}(\Lambda)\ar@<.7ex>[r]\ar@<.7ex>[l]^-{\Gamma}&\ar@<.7ex>[l]^-{\tau_{\geq 0}}\ch(\Lambda)},
\end{equation}
where the first is the ``free-forgetful'' adjunction, and the last is
the obvious adjunction between connective and unbounded chain
complexes involving the good truncation functor $\tau_{\geq 0}$. Endow
the category of simplicial sets with the Bousfield-Kan model structure
for which cofibrations are levelwise injections and weak equivalences
are weak homotopy equivalences, \ie{} isomorphisms on the homotopy
groups. By transfer along the forgetful functor this induces a model
structure on simplicial $\Lambda$-modules, for which the Dold-Kan
correspondence becomes a Quillen equivalence with the projective model
structure on $\ch_{\geq 0}(\Lambda)$ (\ie{} weak equivalences are
quasi-isomorphisms, fibrations are surjections in positive
degrees). 
It is clear that the last adjunction
in~(\ref{eq:dold-kan-factorization}) is Quillen as well.

\begin{pro} \label{dold-kan} The sequence
  in~(\ref{eq:dold-kan-factorization}) induces a Quillen
  adjunction $$(N\Lambda, \Gamma\tau_{\geq 0}) :
  \simp^{\mathrm{op}}\psh({\cal C})\to \U{\cal C}.$$ Here both
  categories are equipped with the projective model structure.
\end{pro}
\begin{proof}
  Consider presheaves on ${\cal C}$ with values in the different
  categories appearing in~(\ref{eq:dold-kan-factorization}). There is
  an induced sequence of adjunctions between these presheaf
  categories, similar to~(\ref{eq:dold-kan-factorization}). If we
  endow each of them with the projective model structure, then each of
  the right adjoint preserves (trivial) fibrations by our discussion
  above.
\end{proof}

\begin{lem}\label{dg-hty-cx}
  Let $K\in\U{\cal C}$ be cofibrant, $K'\in\U{\cal C}$ arbitrary. Then
  \begin{equation*}
    \Gamma\tau_{\geq 0}\dghom(K,K')
  \end{equation*}
  is a (left) homotopy function complex from $K$ to $K'$ (in the sense
  of~\cite[Def.~17.1.1]{hirschhorn:model-cat-loc}).
\end{lem}
\begin{proof}
  Since $K$ is cofibrant, the functor $\bullet\odot
  K:\ch(\Lambda)\to\U{\cal C}$ is left Quillen, with right adjoint
  $\dghom(K,\bullet)$. We know that $\simplex^{\bullet}$ is a (the
  ``standard'') cosimplicial resolution of the terminal object in
  simplicial sets.  By~\cite[Pro.~17.4.16]{hirschhorn:model-cat-loc}, 
  the left homotopy function complex from $K$ to $K'$ is then given by
  \begin{equation*}
    \dghom(N\Lambda(\simplex^{\bullet})\odot K,K')\cong
    \sset(\simplex^{\bullet},\Gamma\tau_{\geq 0}\dghom(K,K'))\cong \Gamma\tau_{\geq 0}\dghom(K,K').
  \end{equation*}
\end{proof}

\begin{cor}\label{dg-hty-classes}
  Let $K,K'\in\U{\cal C}$ and assume that $K$ is cofibrant. For any
  $n\in\Z$, there is a natural isomorphism
  \begin{equation}\label{eq:dg-hty-classes}
    \hom_{\Ho(\U{\cal C})}(K,K'[n])\cong \h_{n}\dghom(K,K').
  \end{equation}
\end{cor}
\begin{proof}
  By~\cite[Pro.~17.7.1]{hirschhorn:model-cat-loc},
  $\pi_{0}\Gamma\tau_{\geq 0}\dghom(K,K'[n])$ is naturally isomorphic
  to the set of homotopy classes from $K$ to $K'[n]$ which is equal to
  the left hand side of~(\ref{eq:dg-hty-classes}), by general
  properties of model categories. But
  \begin{align*}
    \pi_{0}\Gamma\tau_{\geq 0}\dghom(K,K'[n])&\cong\h_{0}\dghom(K,K'[n])\\
    &\cong\h_{n}\dghom(K,K').\qedhere
  \end{align*}
\end{proof}

\begin{lem}\label{hocolim-complexes}
  Let $K\in\simp^{\mathrm{op}}\U{\cal C}$ be a simplicial object in
  $\U{\cal C}$. Then the homotopy colimit
  $\dL\colim_{\simp^{\mathrm{op}}}K$ is given by
  \begin{equation*}
    \tots(K)\simeq\tots(NK).
  \end{equation*}
\end{lem}
\begin{proof}
  The category $\U{\cal C}$ together with the class of
  quasi-isomorphisms and the functor $\tots:\simp^{\mathrm{op}}\U{\cal
    C}\to\U{\cal C}$ defines a ``simplicial descent category'' in the
  sense
  of~\cite{rodriguez-gonzalez:simpl.desc.,rodriguez-gonzales:phd},
  see~\cite[§5.2]{rodriguez-gonzales:phd}. The result for the first
  object now follows
  from~\cite[Thm.~5.1.i]{rodriguez-gonzalez:simpl.desc.}. Since the
  Moore complex and the normalized complexes are homotopy equivalent
  (see Fact~\ref{moore-normalized}), the result for the second object
  follows from this (or
  see~\cite[Rem.~5.2.3]{rodriguez-gonzales:phd}).
\end{proof}

The Moore and normalized complexes also induce functors from
cosimplicial objects to coconnective chain complexes.
\begin{lem}\label{holim-complexes}
  Let $K\in\simp\U{\cal C}$ be a cosimplicial object in $\U{\cal
    C}$. Then the homotopy limit
  $\dR\lim_{\simp}K$ is given by
  \begin{equation*}
    \totp(K)\simeq\totp(NK).
  \end{equation*}
\end{lem}
\begin{proof}
  This can be deduced from the proof of the previous lemma by passing
  to the opposite categories.
\end{proof}

Finally, the following result is often very useful (\eg in~\cite{choudhury-gallauer:iso-galois}).
\begin{lem}\label{uni-cpt-gen}
  The derived category $\Ho(\U{\cal C})$ is compactly generated by
  the representable objects.
\end{lem}
\begin{proof}
  If $\hom_{\Ho(\U{\cal C})}(\Lambda(c),K[n])=0$ for every $c\in{\cal
    C}$ and $n\in\Z$ then this means by Lemma~\ref{dg-hty-classes}
  that $K$ is objectwise acyclic and hence the zero object in the
  derived category.

  Moreover, given a set $(K^{(i)})_{i\in I}$ of objects in $\U{\cal
    C}$ and $c\in{\cal C}$, the canonical morphism
  \begin{equation*}
    \bigoplus_{i}\hom_{\Ho(\U{\cal C})}(\Lambda(c),K^{(i)})\to \hom_{\Ho(\U{\cal C})}(\Lambda(c),\bigoplus_{i}K^{(i)})
  \end{equation*}
  is identified, again by Lemma~\ref{dg-hty-classes}, with
  \begin{equation*}
    \bigoplus_{i}\h_{0}K^{(i)}(c)\to\h_{0}\bigoplus_{i}K^{(i)}(c),
  \end{equation*}
  which is invertible, thus the representable objects are also
  compact.
\end{proof}

\subsection{An example of a left dg Kan extension}
\label{sec:explicit-left-dg}
We would now like to give a more explicit description of the left
dg Kan extension in a specific situation arising in~\cite{choudhury-gallauer:iso-galois}. The setup is as follows: Let ${\cal C}$ be a
small ordinary category, and $\mathcal{B}$ a cocomplete
$\Lambda$-linear category which is tensored over
$\Mod{\Lambda}$. Finally, we are given a functor $\gamma:{\cal
  C}\to\ch(\mathcal{B})$.

First, notice that $\ch(\mathcal{B})$ is canonically a
dg category, and the tensors on ${\cal B}$ induce a tensor
operation of $\ch(\Lambda)$ on $\ch(\mathcal{B})$, by
\begin{equation*} (K\odot B)_{n}=\oplus_{p+q=n}K_{p}\odot B_{q}
\end{equation*} with the usual differentials.

Notice that by considering a presheaf of $\Lambda$-modules as
concentrated in degree 0, we can consider the restriction of
$\gamma^{*}$ to $\psh({\cal C},\Lambda)$, still denoted by
$\gamma^{*}$. The following lemma gives an alternative
characterization of (the underlying functor of) such a left
dg Kan extension.
\begin{lem}\label{kan-char}\mbox{}
  \begin{enumerate}
  \item $\gamma^{*}$ is the composition
    \begin{equation}\label{eq:kan-char-factorization}
      \U{\cal C}\xrightarrow{\ch(\gamma^{*})}\ch(\ch(\mathcal{B}))\xrightarrow{\tots}\ch(\mathcal{B}).
    \end{equation}
  \item Conversely, $\gamma^{*}$ is characterized (up to natural
    isomorphism) by:
    \begin{enumerate}
    \item $\gamma^{*}$ admits a factorization as in~\eqref{eq:kan-char-factorization}.
    \item $\gamma^{*}$ is cocontinuous.
    \item $\gamma^{*}\circ\Lambda(\bullet)\cong \gamma$.
    \end{enumerate}
  \end{enumerate}
\end{lem}
\begin{proof}\mbox{}
  \begin{enumerate}
  \item This follows easily from our definition of the tensor
    operation on $\ch(\mathcal{B})$ together with the fact that
    colimits in $\ch(\mathcal{B})$ are computed degreewise.
  \item We know that $\gamma^{*}$ satisfies the three properties in the
    statement. Conversely, let us prove that they characterize a
    functor $G$ completely (in terms of $\gamma$). By the first property we
    reduce to prove it for a presheaf $K$ concentrated in degree
    0. Then:
    \begin{align*}
      G(K)&\cong G(\int^{c}K(c)_{\mathrm{cst}}\otimes
      \Lambda(c))&&\text{by the Yoneda
        lemma}\\
      &\cong \int^{c}G(K(c)_{\mathrm{cst}}\otimes\Lambda(c))&&\text{by cocontinuity}.
    \end{align*}
  We are thus reduced to show
    \begin{equation*}
      G(K_{\mathrm{cst}}\otimes\Lambda(c))\cong K\odot \gamma(c),
    \end{equation*}
    naturally in modules $K$ and objects $c\in C$. For this
    we can take a functorial exact sequence
    \begin{equation*}
      \oplus_{I_{2}}\Lambda\xrightarrow{\alpha} \oplus_{I_{1}}\Lambda\to K\to 0
    \end{equation*}
    of $\Lambda$-modules, by which we easily reduce to $K$ free using
    the cocontinuity of $G$. Again by cocontinuity we further reduce
    to $K=\Lambda$ and then our contention follows from the third
    property.\qedhere{}
  \end{enumerate}
\end{proof}

\section{Cofibrant replacement}
\label{sec:cof}
Our goal in this section is to resolve functorially any presheaf of
complexes by a cofibrant object made up of representables. It is clear
how to resolve a single presheaf of $\Lambda$-modules, and it is also
not difficult to extend this to bounded below complexes of presheaves
(essentially due to Fact~\ref{bnd}). As the example
in~\cite[2.3.7]{hovey-modelcategories} shows, not every complex of
representables is cofibrant hence naively extending the procedure to
the unbounded case might apriori run into problems. However, we will
show that such problems do not occur.

\subsection{Preliminaries from homological algebra}

Recall the following basic facts in homological algebra. 
\begin{lem}
  \label{qis-sequential-colimit}
  Let ${\cal A}$ be a Grothendieck abelian category, and let $D,
  D':\N\to \ch({\cal A})$ be two diagrams of complexes in ${\cal A}$
  ($\N$ considered as an ordered set). If $g:D\to D'$ is a morphism of
  diagrams of complexes which is objectwise a quasi-isomorphism, then
  also $\varinjlim g$ is a quasi-isomorphism.
\end{lem}
\begin{proof}
  The sequence
  \begin{equation*}
    0\to \oplus_{n\geq 0}D_{n}\to\oplus_{n\geq 0}D_{n}\to \varinjlim
    D\to 0
  \end{equation*}
  is exact, where the second arrow on $D_{n}$ is defined to be
  $\id-D_{n\to n+1}$. Indeed, the only non-trivial part is exactness
  on the left, and for this one notices that the analogous map
  \begin{equation*}
    \oplus_{n=0}^{m}A_{n}\to\oplus_{n=0}^{m+1}A_{n}
  \end{equation*}
  is a mono and hence stays so after taking the limit over $m$ because
  ${\cal A}$ satisfies (AB5).

  $g$ then induces a morphism of short exact sequences of complexes
  and hence a morphism of distinguished triangles in the derived
  category (which exists because ${\cal A}$ is a Grothendieck
  category). It is then clear that the two vertical arrows
  $\oplus_{n}g_{n}$ are isomorphisms in the derived category hence so
  is the third vertical arrow, $\varinjlim g$.
\end{proof}

\begin{lem}
\label{lem:total-complex-qis}
Let ${\cal A}$ be an abelian Grothendieck category and let $C,C'$ be
two bounded below bicomplexes (i.e.  $C_{\bullet,q}=0$ for all $q\ll
0$) in ${\cal A}$, and let $f:C\to C'$ be a morphism of
bicomplexes. If $f_{\bullet,q}:C_{\bullet,q}\to C'_{\bullet,q}$ is a
quasi-isomorphism of complexes for all $q$, then $\tots(f)$ is a
quasi-isomorphism.
\end{lem}
\begin{proof}
  Without loss of generality, $C_{\bullet,q}=0$ for all negative $q$.
  Let $C(n)=C_{\bullet,\leq n}$, $n\geq 0$, be the stupid
  truncation. In other words, $C(n)$ is the subbicomplex of $C$
  satisfying
    \begin{align*}
      C(n)_{p,q}:=
      \begin{cases}
        C_{p,q}&: q\leq n\\
        0&:q> n;
      \end{cases}
    \end{align*} similarly for $C'$ and $f$. We claim that
    $\tots(f(n))$ is a quasi-isomorphism for all $n$. This is
    proved by induction on $n$. For $n=0$ it is true because of our
    assumption on $f$. For the induction step we use the short exact
    sequence
    \begin{equation*}
      0\to\tots(C(n-1))\to\tots(C(n))\to
      C_{\bullet,n}[-n]\to 0
    \end{equation*}
    of complexes in ${\cal A}$. $f$ gives rise to a morphism of short
    exact sequences, where the induction hypothesis for $n-1$ together
    with our assumption on $f$ show that the outer two arrows are
    quasi-isomorphisms. By the 5-lemma also the middle one, \ie{}
    $\tots(f(n))$, is a quasi-isomorphism.

    Now apply the previous lemma to $D_{n}=\tots(C(n))$,
    $D'_{n}=\tots(C'(n))$, and $g_{n}=\tots(f(n))$ to get the
    result. (One uses here that $\tots$ preserves colimits.)
\end{proof}

\subsection{Construction and proof}

Consider the functor category $\psh({\cal C},\Lambda)$. It is a
Grothendieck abelian category. We call an object of $\psh({\cal
  C},\Lambda)$ \emph{semi-representable} 
if it is a small coproduct of representables. An \emph{SR-resolution}
of an object $K\in\psh({\cal C},\Lambda)$ is a complex $K_{\bullet}$
of semi-representables in $\psh({\cal C},\Lambda)$ together with a
quasi-isomorphism of complexes $K_{\bullet}\to S^{0}K$. Similarly one
defines SR-resolutions for complexes in $\psh({\cal C},\Lambda)$. Note
that a bounded below \emph{SR-resolution} is a cofibrant replacement
by Fact~\ref{bnd}.

\begin{lem}
  Objects in $\psh({\cal C},\Lambda)$ possess a functorial SR-resolution;
  more precisely there exists a functor
  \begin{equation*}
    P:\psh({\cal C},\Lambda)\to \U{\cal C}
  \end{equation*}
  together with a natural transformation $P\to S^{0}$ satisfying:
  \begin{itemize}
  \item the components of $P\to S^{0}$ are all bounded below
    SR-resolutions;
  \item $P$ maps the zero morphism to the zero morphism;
  \item $P$ takes injective morphisms to degreewise split injective morphisms.
  \end{itemize}
\end{lem}
\begin{proof}
  Let $K$ be an arbitrary object of $\psh({\cal C},\Lambda)$. There is a canonical
  epimorphism
  \begin{equation*}
    K_{0}:=\bigoplus_{K(c)\backslash 0}\Lambda(c)\to\colim_{K(c)}\Lambda(c)\xlongrightarrow{\sim}K.
  \end{equation*}
  Taking the kernel and repeating this construction we get a complex
  $K_{\bullet}$ together with a quasi-isomorphism $K_{\bullet}\to
  S^{0}K$.

  Given $f: K\to K'$ and $x\in K(c)\backslash 0$ such that $f(x)=0$,
  the component $\Lambda(c)$ corresponding to $x$ is mapped to 0,
  otherwise it maps identically to $\Lambda(c)$ corresponding to
  $f(x)$. It is easily checked that this induces a morphism $\ker(K_0
  \to K) \to \ker(K'_0 \to K')$ hence repeating we obtain
  $P(f):P(K)\to P(K')$. Functoriality is clear.

  If $f$ is injective then by this description $f_0: K_0 \to K'_0$ is
  split injective, and the induced morphism $\ker(K_0 \to K) \to
  \ker(K'_0 \to K')$ is injective. Repeating this argument, we see
  that the induced morphism $P(f)$ is degreewise split injective.
\end{proof}

\begin{pro}\label{cof-res}
  There exists an endofunctor $Q:\U{\cal C}\to\U{\cal C}$ together
  with a natural transformation $Q\to\id$ satisfying:
  \begin{itemize}
  \item the components of $Q\to\id$ are trivial fibrations;
  \item the image of $Q$ consists of projective cofibrant complexes of
    semi-representables.
  \end{itemize}
  In particular, $Q$ is a cofibrant replacement functor.
\end{pro}
\begin{proof}
  Apply the functor $P$ of the previous lemma in each degree,
  obtaining an SR-resolution $P(K_{n})$ of $K_{n}$ for each
  $n\in\Z$. We get maps $P(K_{n})\to P(K_{n-1})$
  of complexes which in total define a bicomplex
  $P(K):=P_{\bullet}(K_{\bullet})$ (since $P$ takes 0 to 0) together
  with a map of bicomplexes $P(K)\to K$, the latter concentrated in
  horizontal degree 0. Taking the total complexes yields a morphism
  \begin{equation}\label{eq:cof-res}
    Q(K):=\tots(P_{\bullet}(K_{\bullet}))\to \tots(K_{\bullet})=K.
  \end{equation}
  Functoriality follows from functoriality in the previous lemma as
  well as functoriality of $\tots$. It remains to prove
  that~(\ref{eq:cof-res}) is a quasi-isomorphism with projective
  cofibrant domain.

  For this let $\tau_{\geq n} K$ ($n\in\Z$) be the subcomplex of $K$
  satisfying
  \begin{align*}
    (\tau_{\geq n} K)_q=
    \begin{cases}
      K_q &: q > n\\
      Z_{n}K &: q = n\\
      0&:q < n.
    \end{cases}
  \end{align*}
  Note that there are canonical morphisms $\tau_{\geq n} K \to
  \tau_{\geq n-1} K$ and the canonical morphism $\varinjlim_{n \in \N}
  \tau_{\geq -n} K \to K$ is an isomorphism. But also
  $\varinjlim_{n\in\N}P(\tau_{\geq -n}K)\to P(K)$ is an isomorphism of
  bicomplexes. Since the total complex functor commutes with colimits
  we conclude that $\varinjlim_{n\in\N}Q(\tau_{\geq -n}K)\to Q(K)$ is
  an isomorphism.

  By the previous lemma, $P(\tau_{\geq -n}K)\to P(\tau_{\geq
    -(n+1)}K)$ is a bidegreewise split injection hence $Q(\tau_{\geq
    -n}K)\to Q(\tau_{\geq -(n+1)}K)$ is a degreewise split
  injection. It follows from Corollary~\ref{cof-un} that $Q(K)$ is
  projective cofibrant. Also by the previous lemma, $P(\tau_{\geq
    -n}K)\to\tau_{\geq -n}K$ is a quasi-isomorphism in each row. It
  follows from Lemma~\ref{lem:total-complex-qis} that $Q(\tau_{\geq
    -n}K)\to\tau_{\geq -n}K$ is a
  quasi-isomorphism. (\ref{eq:cof-res}) being the sequential colimit
  of these morphisms, Lemma~\ref{qis-sequential-colimit} tells us that
  also~(\ref{eq:cof-res}) is a quasi-isomorphism.
\end{proof}

\begin{rem}
  Even if this result is not very useful from a practical point of
  view, it does provide a conceptually satisfying method to compute
  the derived functor of a left dg Kan extension in the
  context of §\ref{sec:explicit-left-dg}. Indeed, fix a functor
  $\gamma:{\cal C}\to \ch({\cal B})$ for a $\Mod{\Lambda}$-cocomplete
  $\Lambda$-linear category ${\cal B}$, and assume that $\gamma^{*}$
  is a left Quillen functor. The image of any $K\in\U{\cal C}$ under
  $\dL \gamma^{*}$ can be computed as follows:
  \begin{enumerate}
  \item Resolve $K$ by a cofibrant complex $QK$ of
    semi-representables.
  \item Apply $\gamma$ to each representable in $QK$ obtaining a
    bicomplex $\gamma(QK)$ in ${\cal B}$.
  \item Take the total complex $\tots(\gamma(QK))$.
  \end{enumerate}
  In particular, this provides a more ``elementary'' description of
  the motivic realization constructed
  in~\cite[\S7]{choudhury-gallauer:iso-galois}.
\end{rem}

\section{Local model structures}
\label{sec:local}
Having dealt with ``generators'' for universal enriched homotopy
theories in §\ref{sec:uni-psh} and for universal dg homotopy theories
in more detail in the subsequent sections, we now turn to
``relations''. The only sort of relations we will be interested in
here are ``topological'', \ie{} induced by a Grothendieck topology on
${\cal C}$. Unfortunately we are not able to prove any substantial
facts in the general enriched setting which is why we again restrict
to the case of dg categories. Here, our main result is completely
analogous to the main result of~\cite{dugger-hollander-isaksen} where
it is shown that a simplicial presheaf in the Jardine local model
structure is fibrant if and only if it is injective fibrant and
satisfies descent with respect to hypercovers.

Throughout this section we assume that ${\cal C}$ is endowed with a
Grothendieck topology $\tau$. Let $\sh_{\tau}({\cal C})$ (resp.\
$\sh_{\tau}({\cal C}, \Lambda)$) denote the category of $\tau$-sheaves
(resp.\ of sheaves of $\Lambda$-modules) on ${\cal C}$. The embedding
$\sh_{\tau}({\cal C}) \to \psh({\cal C})$ (resp.\ $\sh_{\tau}({\cal
  C}, \Lambda) \to \psh({\cal C}, \Lambda)$) is right adjoint to the
sheafification functor $a_{\tau}$.

\subsection{Hypercovers and descent}

Recall (\cite[§3]{dugger-hollander-isaksen}) that a morphism $f$ of
presheaves is a \emph{generalized cover} if its sheafification
$a_{\tau}(f)$ is an epimorphism.

\begin{dfi}
  For any object $c \in {\cal C}$ a \emph{$\tau$-hypercover of $c$} is
  a simplicial presheaf of sets $c_{\bullet}$ on ${\cal C}$ with an
  augmentation map $c_{\bullet} \to c=:c_{-1}$ such that
  \begin{itemize}
  \item $c_n$ is a coproduct of representables for all $n\in\N$, and
  \item $c_{n}\to (\mathrm{cosk}_{n-1}\mathrm{sk}_{n-1}c_{\bullet})_{n}$
    is a generalized cover for all $n\in\N$.
  \end{itemize}
  (To avoid any confusion, the cases $n=0,1$ of the second bullet
  point require $c_{0}\to c$ and $d_{0}\times d_{1}:c_{1}\to
  c_{0}\times_{c}c_{0}$ to be generalized covers, respectively.) A
  \emph{refinement} of a hypercover $c_{\bullet}\to c$ is a hypercover
  $c'_{\bullet}\to c$ together with a morphism of simplicial
  presheaves $c'_{\bullet}\to c_{\bullet}$ compatible with the
  augmentation by $c$. The class of all $\tau$-hypercovers of $c$ is
  denoted by ${\cal H}_{\tau,c}$. Also set ${\cal
    H}_{\tau}:=\coprod_{c\in{\cal C}}{\cal H}_{\tau,c}$. A subclass
  ${\cal H}$ of ${\cal H}_{\tau}$ (resp.\ ${\cal H}_{\tau,c}$) is
  called \emph{dense} if every $\tau$-hypercover (resp.\ of $c$)
  admits a refinement by a hypercover in ${\cal H}$.
\end{dfi}
We refer to~\cite{dugger-hollander-isaksen} for details about
hypercovers. In particular, we recall without proof the following
important fact.
\begin{fac}[{\cite[Pro.~6.7]{dugger-hollander-isaksen}}]\label{dense-subset}
  For every $c\in{\cal C}$, there exists a dense subset of ${\cal
    H}_{\tau,c}$. Therefore also ${\cal H}_{\tau}$ admits a dense
  subset.
\end{fac}

In the case of simplicial presheaves the $\tau$-hypercovers provide
the ``topological'' relations in that the hypercover $c_{\bullet}$ and
the representable $c$ are ``identified'', and we want to translate
these relations to the setting of presheaves of complexes. For this
notice that given any hypercover $c_{\bullet}\to c$ we can use the
Moore complex (cf.~§\ref{sec:dold-kan}) to obtain an object
$\Lambda(c_{\bullet})\in\U{\cal C}$ together with a morphism
$\Lambda(c_{\bullet})\to\Lambda(c)$. Explicitly,
$\Lambda(c_{\bullet})$ is the complex
\begin{equation*}
  \cdots\to\Lambda(c_{1})\to\Lambda(c_{0})\to 0
\end{equation*}
with differentials given by the alternating sum of the face maps, and
each $\Lambda(c_{i})$ is semi-representable. It follows from
Lemma~\ref{bnd} that $\Lambda(c_{\bullet})$ is projective cofibrant.

\begin{dfi}\mbox{}
  \begin{enumerate}
  \item Let ${\cal S}$ be a class of $\tau$-hypercovers. A presheaf
    $K\in\U{\cal C}$ \emph{satisfies ${\cal S}$-descent} if for any
    $\tau$-hypercover $c_{\bullet}\to c$ in ${\cal S}$,
    \begin{equation*}
      K(c)=\dghom(\Lambda(c),K)\to \dghom(\Lambda(c_{\bullet}),K)=:K(c_{\bullet})
    \end{equation*}
    is a quasi-isomorphism of chain complexes.
  \item $K\in\U{\cal C}$ \emph{satisfies $\tau$-descent} if it
    satisfies ${\cal H}_{\tau}$-descent.
  \end{enumerate}
\end{dfi}
Explicitly, $K(c_{\bullet})$ is given by the product total complex of
the bicomplex
\begin{equation*}
  K(c_{0})\to K(c_{1})\to\cdots,
\end{equation*}
where $K(\coprod_{i\in I}d_{i})$ for $d_{i}\in{\cal C}$ is defined as
$\prod_{i\in I}K(d_{i})$.
\begin{rem}
  The condition of satisfying descent is homotopy invariant, \ie{} given
  two quasi-isomorphic presheaves of complexes, one satisfies ${\cal
    S}$-descent if and only if the other does. Indeed, as we know from
  Fact~\ref{uni-model-v-cat}, $\dghom:(\U{\cal
    C})^{\mathrm{op}}\times\U{\cal C}\to\ch(\Lambda)$ is part of a
  Quillen adjunction of two variables. And since every object in
  $\U{\cal C}$ is fibrant, and since both $\Lambda(c_{\bullet})$ and
  $\Lambda(c)$ are cofibrant (by Fact~\ref{bnd}), the condition on $K$
  to satisfy descent is that
  \begin{equation*}
    \dR\dghom(\Lambda(c),K)\to\dR\dghom(\Lambda(c_{\bullet}),K)
  \end{equation*}
  be an isomorphism in the derived category of $\Lambda$. This is
  different from the situation of simplicial presheaves of sets where
  $c_{\bullet}$ is not necessarily projective cofibrant. Thus the
  interest in \emph{split} hypercovers,
  cf.~\cite[Cor.~9.4]{dugger-universal-hty}.
\end{rem}

We end this section by the following important result. In terminology
to be introduced shortly it tells us that the augmentation morphism
$\Lambda(c_{\bullet})\to \Lambda(c)$ associated to any
$\tau$-hypercover is a $\tau$-local equivalence.

\begin{fac}[{\cite[Thm.~V,~7.3.2]{sga4}}] \label{hypercover-acyclic}
  Any $\tau$-hypercover $c_{\bullet} \to c$ induces identifications
  \begin{equation*}
    a_{\tau}\h_{n}\Lambda(c_{\bullet})\cong
    \begin{cases}
      a_{\tau}\Lambda(c)&:n=0\\
      0&:n\neq 0
    \end{cases}
  \end{equation*}
  in $\sh({\cal C},\Lambda)$.
\end{fac}

\subsection{Localization}

\begin{dfi}
  A morphism $f$ in $\U{\cal C}$ is called a \emph{$\tau$-local
    equivalence} if the induced morphism of homology sheaves
  $a_{\tau}\h_n(f)$ is an isomorphism for all $n \in \Z$.
\end{dfi}

The goal of this section is to prove the following theorem.
\begin{thm}\label{local-model-structure}
  The left Bousfield localization $\U{\cal C}/\tau$ of $\U{\cal C}$
  with respect to $\tau$-local equivalences exists and satisfies:
  \begin{enumerate}
  \item The underlying category of $\U{\cal C}/\tau$ is the one of
    $\U{\cal C}$. The cofibrations are also the same. The weak
    equivalences are the $\tau$-local equivalences.
  \item $\U{\cal C}/\tau$ is a proper, tractable, cellular, stable
    model category.
  \item The fibrations of $\U{\cal C}/\tau$ are the fibrations of
    $\U{\cal C}$ whose kernel satisfies $\tau$-descent. In particular,
    the fibrant objects of $\U{\cal C}/\tau$ are the objects
    satisfying $\tau$-descent.\label{local-fibrations}
  \end{enumerate}
  The model structure on $\U{\cal C}/\tau$ is called the
  \emph{$\tau$-local model structure}.
\end{thm}

\begin{rem}\label{existence-local-model-structure}
  This result was originally one of our main motivations to write the
  present note. The existence of this localization was known
  before, see~\cite[Def.~4.4.34]{ayoub07-thesis}, and we use this
  result in our proof. The main point of the theorem for us was
  part~\ref{local-fibrations}. The analogous description of the
  fibrant objects for simplicial sets instead of chain complexes is of
  course the main result of~\cite{dugger-hollander-isaksen}, and we
  deduce our result from theirs.

  After having completed this note, we learned that also
  part~\ref{local-fibrations} had appeared in the literature before,
  see~\cite{hinich:def-sh-alg}. His proof is different from ours in
  that he does not reduce to the case of simplicial sets nor uses the
  theory of Bousfield localizations but proves the axioms of a model
  structure ``by hand''.
\end{rem}

Let ${\cal S}\subset{\cal H}_{\tau}$ be some class of
$\tau$-hypercovers. We denote by $\Lambda({\cal S})[\Z]$ the class
\begin{equation*}
  \{\Lambda(c_{\bullet})[n]\to
  \Lambda(c)[n]\mid c_{\bullet}\to c\in{\cal S}, n\in\Z\}
\end{equation*}
of morphisms in $\U{\cal C}$.

\begin{dfi}\mbox{}
  \begin{enumerate}
  \item Recall (\cite[Def.~3.1.4]{hirschhorn:model-cat-loc}) that an
    object $K$ in $\U{\cal C}$ is called \emph{local} with respect to
    a class of morphisms ${\cal F}$ in $\U{\cal C}$ if for each $f\in
    {\cal F}$, the induced morphism of homotopy function complexes
    $\hmap(f,K)$ is a weak homotopy equivalence of simplicial
    sets.
  \item Let ${\cal S}$ be a class of $\tau$-hypercovers. We say that
    $K\in\U{\cal C}$ is \emph{${\cal S}$-local} if it is local with
    respect to $\Lambda({\cal S})[\Z]$.
  \item We say that $K\in\U{\cal C}$ is \emph{$\tau$-local} if it is
    ${\cal H}_{\tau}$-local.
  \end{enumerate}
\end{dfi}

\begin{lem}\label{local=descent}
  For a presheaf of complexes $K \in \U{\cal C}$ and a class ${\cal
    S}$ of $\tau$-hypercovers the following two conditions are
  equivalent:
  \begin{enumerate}
  \item $K$ is ${\cal S}$-local.
  \item $K$ satisfies ${\cal S}$-descent.
  \end{enumerate}
  In particular, the following two conditions are equivalent:
  \begin{enumerate}
  \item $K$ is $\tau$-local.
  \item $K$ satisfies $\tau$-descent.
  \end{enumerate}
\end{lem}
\begin{proof}
  $K$ is ${\cal S}$-local if and only if for any $c_{\bullet}\to c\in{\cal S}$,
  $n\in\Z$, the morphism of homotopy function complexes
  \begin{equation}\label{eq:local-hypercover}
    \hmap(\Lambda(c)[n],K)\to \hmap(\Lambda(c_{\bullet})[n],K)
  \end{equation}
  is a weak equivalence of simplicial sets. But
  $\hmap(A,B)\cong \Gamma\tau_{\geq 0}\U{\cal C}(A,B)$ by
  Lemma~\ref{dg-hty-cx}. So~(\ref{eq:local-hypercover}) is identified
  with
  \begin{equation*}
    \Gamma\tau_{\geq -n}K(c)\to \Gamma\tau_{\geq -n}K(c_{\bullet}),
  \end{equation*}
  whose $m$-th homotopy group is thus
  \begin{equation*}
    \h_{m-n}K(c)\to \h_{m-n}K(c_{\bullet}).\qedhere
  \end{equation*}
\end{proof}

We will deduce Theorem~\ref{local-model-structure} from the following
(cf.~\cite[Thm.~6.2]{dugger-hollander-isaksen}).
\begin{thm}\label{hypercover-localization}
  Let ${\cal S}$ be a class of $\tau$-hypercovers which contains a
  dense subset. Then the left Bousfield localization $\U{\cal C}/{\cal
    S}$ of $\U{\cal C}$ with respect to $\Lambda({\cal S})[\Z]$
  exists and coincides with $\U{\cal C}/\tau$.
\end{thm}
\begin{proof}[Proof of Theorem~\ref{local-model-structure}]Let ${\cal
    S}$ be the class of all $\tau$-hypercovers. By
  Fact~\ref{dense-subset}, ${\cal S}$ satisfies the assumption of
  Theorem~\ref{hypercover-localization}. We know from
  Corollary~\ref{dg-model-prop} that $\U{\cal C}$ is left-proper,
  tractable, cellular. These are preserved by left Bousfield
  localizations
  by~\cite[Thm.~4.1.1]{hirschhorn:model-cat-loc} 
  and~\cite[Pro.~4.3]{hovey:comodules-hopf-algebroid}. Since ${\cal
    S}$-local objects are closed under shifts by
  Lemma~\ref{local=descent}, $\U{\cal C}/{\cal S}$ and therefore
  $\U{\cal C}/\tau$ are stable model categories
  (see~\cite[Pro.~3.6]{barnes-roitzheim:bousfield-loc}). Since
  $\U{\cal C}$ is a right proper model category so is $\U{\cal
    C}/\tau$ by~\cite[Pro.~3.7]{barnes-roitzheim:bousfield-loc}. Since
  all objects are fibrant in $\U{\cal C}$, the fibrant objects of
  $\U{\cal C}/{\cal S}$ are the $\tau$-local objects. We deduce from
  Lemma~\ref{local=descent} and Theorem~\ref{hypercover-localization}
  that the fibrant objects of $\U{\cal C}/\tau$ are precisely the
  presheaves satisfying $\tau$-descent. The description of the
  fibrations in $\U{\cal C}/\tau$ then follows from this
  and~\cite[Lem.~3.9]{barnes-roitzheim:bousfield-loc}. Finally, that
  the weak equivalences of $\U{\cal C}/\tau$ are the $\tau$-local
  equivalences is proven in~\cite[Pro.~4.4.32]{ayoub07-thesis}.
\end{proof}

Assume for the moment that ${\cal S}$ in
Theorem~\ref{hypercover-localization} is a set. In this case we know
that the left Bousfield localization $\U{\cal C}/{\cal S}$ (resp.\
$\simp^{\mathrm{op}}\psh({\cal C})/{\cal S}$) with respect to
$\Lambda({\cal S})[\Z]$ (resp.\ ${\cal S}$) exists. Temporarily, we
call these model structures the ${\cal S}$-local model structures,
their fibrations are called ${\cal S}$-fibrations, their weak
equivalences are called ${\cal S}$-equivalences.
\begin{lem}\label{dold-kan-local}
  The Dold-Kan correspondence (Proposition~\ref{dold-kan}) induces a
  Quillen adjunction
  \begin{equation*}
    (N\Lambda,\Gamma\tau_{\geq 0}):\simp^{\mathrm{op}}\psh({\cal
      C})/{\cal S}\longrightarrow \U{\cal C}/{\cal S}.
  \end{equation*}
  Moreover, $\Gamma\tau_{\geq 0}$ preserves $\tau$-local equivalences.
\end{lem}
\begin{proof}
  Given $f:c_{\bullet}\to c\in{\cal S}$, the morphism $N\Lambda(f)$
  factors as
  \begin{equation*}
    N\Lambda(c_{\bullet})\to\Lambda(c_{\bullet})\xrightarrow{\Lambda(f)} \Lambda(c),
  \end{equation*}
  where the first arrow is a quasi-isomorphism by
  Fact~\ref{moore-normalized}, and the second arrow lies in
  $\Lambda({\cal S})[\Z]$. Thus the first claim follows from the
  universal property of localizations. The second claim is also
  evident since the homotopy groups of $\Gamma \tau_{\geq 0}K$ are the
  homology groups of $K$ in non-negative degrees, by
  Fact~\ref{moore-normalized}.
\end{proof}

\begin{lem}\label{descent-to-local}
  Let $K, K' \in \U{\cal C}$ be ${\cal S}$-fibrant objects and let $f
  : K \to K'$ be an ${\cal S}$-fibration which is also a $\tau$-local
  weak equivalence. Then $f$ is a sectionwise trivial fibration, \ie{}
  it is a trivial fibration in the projective model structure on
  $\U{\cal C}$.
\end{lem}
\begin{proof}

  A morphism $f : K \to K' \in \U{\cal C}$ is a trivial fibration if
  and only if for all $c \in {\cal C}$ and all $n \in \Z$, $f$ has the
  right lifting property with respect to $\partial\simplex^{n}\Lambda(c) \to
  \simplex^n \Lambda(c)$ (see Lemma \ref{I'}).

  Let $i : (\partial \simplex^1) \otimes c \to \simplex^1 \otimes c$
  be the canonical cofibration of simplicial presheaves. Then
  $N\Lambda(i)$ is also a cofibration and $N\Lambda((\partial
  \simplex^1) \otimes c) = \partial\simplex^1 \Lambda(c)$ and
  $N\Lambda(\simplex^1 \otimes c) = \simplex^1 \Lambda(c)$. Also note
  that $\partial\simplex^n \Lambda(c) = \partial\simplex^1 \Lambda(c)
  [-n+1]$, and similarly for $\simplex^n \Lambda(c)$.  We want to show
  the existence of a lifting for every diagram of the following form
  $$\xymatrix{\partial\simplex^{n} \Lambda(c)\ar[r] \ar[d] & 
    K\ar[d]\\
    \simplex^n \Lambda(c) \ar[r] & K'}$$
  Now using shifts this is the same as showing that
  $$\xymatrix{\partial\simplex^1\Lambda(c)\ar[r] \ar[d] & 
    K[n-1]\ar[d]\\
    \simplex^1\Lambda(c) \ar[r] & K'[n-1]}$$ has a lift. Notice that
  the right vertical arrow is still an ${\cal S}$-fibration.

  But using the adjunction of Lemma~\ref{dold-kan-local} this is the
  same as showing that
  $$\xymatrix{\partial\simplex^{1}\otimes c\ar[r] \ar[d] & 
    \Gamma\tau_{\geq 0}(K[n-1])\ar[d]\\
    \simplex^{1}\otimes c \ar[r] & \Gamma\tau_{\geq 0}(K'[n-1])}$$
  has a lift, where we know that the right vertical arrow is an ${\cal
    S}$-fibration and $\tau$-local equivalence between ${\cal
    S}$-fibrant objects. Hence by
  \cite[Lem.~6.5]{dugger-hollander-isaksen} it is a trivial fibration
  sectionwise. Now $i : (\partial \simplex^1) \otimes c \to
  \simplex^1 \otimes c$ is a projective cofibration, hence there is a
  lift in the last diagram above. This finishes the proof.
\end{proof}

\begin{proof}[Proof of Theorem~\ref{hypercover-localization}]
  Let ${\cal S}$ be as in the theorem, and pick a dense subset ${\cal
    S}'$ of ${\cal S}$.

  We claim that the ${\cal S}'$-local equivalences in $\U{\cal C}$ are
  precisely the $\tau$-local equivalences. Indeed, by
  Fact~\ref{hypercover-acyclic}, every ${\cal S}'$-local equivalence is a
  $\tau$-local equivalence. For the converse, we may
  apply~\cite[Lem~6.4]{dugger-hollander-isaksen} together with
  Lemma~\ref{descent-to-local} (we also use the existence of the
  $\tau$-local model structure, see
  Remark~\ref{existence-local-model-structure}). This proves the claim
  which in turn implies that $\U{\cal C}/{\cal S}'=\U{\cal C}/\tau$.

  We deduce that every hypercover in ${\cal S}$ is an ${\cal
    S}'$-equivalence hence the localization of $\U{\cal C}$ with
  respect to $\Lambda({\cal S})[\Z]$ exists and coincides with
  $\U{\cal C}/{\cal S}'$.
\end{proof}

Let us agree to call a model category ${\cal M}$ equipped with a
functor $\gamma:{\cal C}\to{\cal M}$ \emph{$\tau$-local} if for every
$\tau$-hypercover $c_{\bullet}\to c$ in ${\cal C}$,
$\dL\colim_{\simp^{\mathrm{op}}}\gamma(c_{\bullet})\to \gamma(c)$ is
an isomorphism in $\Ho({\cal M})$. In line with the viewpoint taken in
§\ref{sec:uni-psh} let us record the following corollary of
Theorem~\ref{local-model-structure}. It asserts that $\Udg{\cal
  C}/\tau$ is the universal $\tau$-local model dg category
associated to ${\cal C}$.

\begin{cor}\label{universal-local-dg-model}
  Let $({\cal C},\tau)$ be a small site. Then there exists a functor
  $\Lambda:{\cal C}\to\Udg{\cal C}/\tau$ into a $\tau$-local model
  dg category, universal in the sense that for any solid
  diagram
  \begin{equation*}
    \xymatrix{{\cal C}\ar[rd]_{\gamma}\ar[r]^-{\Lambda}&\Udg{\cal C}/\tau\ar@{.>}[d]^{F}\\
      &{\cal M}}
  \end{equation*}
  with ${\cal M}$ a $\tau$-local model dg category, there
  exists a left Quillen dg functor $F$ as indicated by the
  dotted arrow, unique up to a contractible choice, making the diagram
  commutative up to a weak equivalence $F\Lambda\to\gamma$.
\end{cor}
\begin{proof}
  $\Udg{\cal C}/\tau$ as a dg category is just $\Udg{\cal C}$ and the
  cofibrations are the same hence to prove that $\Udg{\cal C}/\tau$ is
  a model dg category, it suffices to see that the pushout-product
  $i\square f$ is a $\tau$-weak equivalence for every cofibration $i$
  in $\ch(\Lambda)$ and every $\tau$-acyclic cofibration $f\in\U{\cal
    C}$. This can be established exactly as in the proof
  of~\cite[Thm.~4.46]{barwick-modcat-bousfield}. (For this step it is
  not necessary to assume as is done in loc.\, cit.\ that the
  localization is with respect to a set but only that it exists.)  The
  essential point is that $\U{\cal C}$ is a tractable model category
  (by~Proposition~\ref{dg-model-prop}).

  Next we claim that
  $\dL\colim_{\simp^{\mathrm{op}}}\Lambda(c_{\bullet})\to \Lambda(c)$
  is an isomorphism in $\Ho(\Udg{\cal C}/\tau)$. But by
  Lemma~\ref{hocolim-complexes}, this morphism can be identified with
  $\Lambda(c_{\bullet})\to\Lambda(c)$ hence the claim follows from
  Lemma~\ref{hypercover-acyclic}.

  Given a solid diagram as in the statement of the corollary we know
  by Corollary~\ref{enriched-model-universal} the existence of a left
  Quillen dg functor $F:\Udg{\cal C}\to{\cal M}$, unique up
  to contractible choice, making the triangle commutative up to a weak
  equivalence $Fy\to\gamma$. By the universal property of the
  localization of model categories together with
  Theorem~\ref{hypercover-localization}, it now suffices to prove that
  the left derived functor $\dL F$ takes $\Lambda({\cal
    H}_{\tau})[\Z]$ to isomorphisms in $\Ho({\cal M})$. Thus let
  $c_{\bullet}\to c\in{\cal H}_{\tau}$ and $n\in\Z$. First notice that
  $F$ ``commutes with shifts'' in the sense that
  \begin{equation*}
    F(\bullet[n])\cong F(S^{n}\odot \bullet)\cong S^{n}\odot F(\bullet),
  \end{equation*}
  and since ${\cal M}$ is a model dg category,
  $S^{n}\odot\bullet$ preserves weak equivalences. We thus reduce to
  the case $n=0$.

  Now, again by Lemma~\ref{hocolim-complexes},
  $\Lambda(c_{\bullet})$ can be identified with the homotopy
  colimit of $\Lambda(c_{\bullet})$. Since $F$ is a left Quillen
  dg functor it will commute with homotopy colimits in the
  homotopy category. Thus we want the upper row in the following
  commutative square to be invertible in $\Ho({\cal M})$.
  \begin{equation*}
    \xymatrix{\dL\colim_{\simp^{\mathrm{op}}}F\Lambda(c_{\bullet})\ar[r]\ar[d]&F\Lambda(c)\ar[d]\\
      \dL\colim_{\simp^{\mathrm{op}}}\gamma(c_{\bullet})\ar[r]&\gamma(c)}
  \end{equation*}
  Our assumptions tell us that the vertical arrows as well as the
  bottom arrow are isomorphisms so we conclude.
\end{proof}

\subsection{Smaller models}
Having described explicitly generators and relations for the model dg
category $\U{\cal C}/\tau$ associated to a small site $({\cal
  C},\tau)$, we give in this section two methods to modify the model
$\U{\cal C}/\tau$ up to Quillen equivalence which are useful in
practice. The first consists in replacing presheaves by sheaves, the
second allows to reduce the ``number'' of generators. In both cases
therefore we obtain ``smaller'' models with the same homotopy
category. Both modifications are straightforward and have been
employed before in the literature.

The category of $\tau$-sheaves of complexes on ${\cal C}$,
$\sh_{\tau}({\cal C},\ch(\Lambda))$, admits the $\tau$-local model
structure, obtained by transfer along the right adjoint
$\sh_{\tau}({\cal C},\ch(\Lambda))\to\psh({\cal C},\ch(\Lambda))$
(cf.~\cite[Cor.~4.4.43]{ayoub07-thesis}). Since the morphism $K\to
a_{\tau}K$ is a $\tau$-local equivalence for every $K\in\U{\cal C}$,
the following statement is immediate.

\begin{fac}
  \begin{equation*}
    \xymatrix{\U{\cal C}/\tau\ar@<.5ex>[r]^-{a_{\tau}}&\sh_{\tau}({\cal C},\ch(\Lambda))/\tau\ar@<.5ex>[l]}
  \end{equation*}
  defines a Quillen equivalence. Their homotopy categories are the
  derived category of $\tau$-sheaves on ${\cal C}$.
\end{fac}

It happens frequently that every object $c\in{\cal C}$ can be covered
by objects belonging to a distinguished strict subcategory ${\cal
  C}'$. Certainly one then expects the model dg categories
generated by ${\cal C}$ and ${\cal C}'$ with the topological relations
to be ``the same''. The following result makes this precise.
\begin{cor}\label{subsite}
  Let ${\cal C}'$ be a full subcategory of ${\cal C}$, and endow it
  with the topology $\tau'$ induced from $\tau$. Assume that every
  object $c\in{\cal C}$ can be covered by objects belonging to ${\cal
    C}'$. Then the (functor underlying the) canonical
  dg functor
  \begin{equation*}
    \U{\cal C}'/\tau'\longrightarrow \U{\cal C}/\tau
  \end{equation*}
  defines a Quillen equivalence.
\end{cor}
\begin{proof}
  The composition ${\cal C}'\xrightarrow{u}{\cal C}\to\U{\cal C}/\tau$
  induces the left Quillen dg functor $u_{!}$ in the
  statement by the universal property of $\U{\cal C}'/\tau'$
  (Corollary~\ref{universal-local-dg-model}), left-adjoint to the
  restriction functor $u^{*}$. 
    Consider
    the square of Quillen right functors:
    \begin{equation*}
      \xymatrix{\sh_{\tau}({\cal
          C},\ch(\Lambda))/\tau\ar[r]^{u^{*}}\ar[d]&\sh_{\tau'}({\cal
          C}',\ch(\Lambda))/\tau'\ar[d]\\
        \U{\cal C}/\tau\ar[r]_{u^{*}}&\U{\cal C}'/\tau'}
    \end{equation*}
    Clearly, it commutes. By the previous fact, the vertical arrows
    are part of a Quillen equivalence, and the homotopy categories in
    the top row are the derived categories of $\tau$-sheaves (resp.\
    $\tau'$-sheaves) on ${\cal C}$ (resp.\ ${\cal
      C}'$). By~\cite[Thm.~III.4.1]{sga4}, the top arrow is an
    equivalence of the underlying categories hence so is the induced
    functor on their derived categories.
\end{proof}

\subsection{Hypercohomology}

One might hope that the results obtained so far in this section allow
to describe a $\tau$-fibrant replacement directly in terms of
hypercovers. In particular, this would lead to an expression for the
hypercohomology of complexes of sheaves using hypercovers alone. We
have not been able to provide such a fibrant replacement but, as we
will now show, the hypercohomology does indeed admit such an expected
description. This result should be compared to Verdier's hypercover
theorem in~\cite[Thm.~V,~7.4.1]{sga4}. Our proof once again proceeds
by reducing to the case of simplicial (pre)sheaves of sets
in~\cite{dugger-hollander-isaksen}. (In the following, we write
$\h^{n}$ for $\h_{-n}$.)
\begin{pro}
  Assume that every $\tau$-hypercover can be refined by a split
  one. Let $K\in\U{\cal C}$ be a presheaf of complexes on ${\cal C}$,
  $c\in{\cal C}$, and $n\in\Z$. Then there is a canonical isomorphism
  of $\Lambda$-modules
  \begin{equation*}
    \mathbb{H}^{n}_{\tau}(c,a_{\tau}K)\cong \colim_{c_{\bullet}\to c}\h^{n}K(c_{\bullet}),
  \end{equation*}
  where the left hand side denotes hypercohomology of the complex of
  $\tau$-sheaves $a_{\tau}K$ on ${\cal C}/c$, and the colimit on the
  right hand side is over the opposite category of $\tau$-hypercovers
  of $c$ up to simplicial homotopy (cf.~\cite[§V.7.3]{sga4}).
\end{pro}
\begin{proof}
  This follows from the following sequence of isomorphisms:
  \begin{align*}
    \mathbb{H}^{n}_{\tau}(c,a_{\tau}K)&\cong \hom_{\Ho(\U{\cal
        C}/\tau)}(\Lambda(c),a_{\tau}K[-n])&&\text{Corollary~\ref{dg-hty-classes}}\\
    &\cong \hom_{\Ho(\U{\cal
        C}/\tau)}(\Lambda(c),K[-n])&&K\to a_{\tau}K\text{ $\tau$-local
      equivalence}\\
    &\cong\hom_{\Ho(\simp^{\mathrm{op}}\psh({\cal
        C})/\tau)}(c,\Gamma\tau_{\geq
      -n}K)&&\text{Lemma~\ref{dold-kan-local}}\\
    &\cong\colim_{c_{\bullet}\to c}\pi(c_{\bullet},\Gamma\tau_{\geq
      -n}K)&&\text{\cite[Thm.~7.6(b)]{dugger-hollander-isaksen}}\\
    &\cong\colim_{c_{\bullet}\to c\text{ split}}\pi(c_{\bullet},\Gamma\tau_{\geq
      -n}K)&&\text{assumption}\\
    &\cong\colim_{c_{\bullet}\to c\text{ split}}\hom_{\Ho(\simp^{\mathrm{op}}\psh({\cal C}))}(c_{\bullet},\Gamma\tau_{\geq
      -n}K)&&\text{split hypercovers cofibrant}\\
    &\cong\colim_{c_{\bullet}\to c\text{ split}}\hom_{\Ho(\U{\cal
        C})}(N\Lambda(c_{\bullet}),K[-n])&&\text{Lemma~\ref{dold-kan}}\\
    &\cong\colim_{c_{\bullet}\to c\text{ split}}\hom_{\Ho(\U{\cal
        C})}(\Lambda(c_{\bullet}),K[-n])&&\text{Lemma~\ref{moore-normalized}}\\
    &\cong\colim_{c_{\bullet}\to c\text{
        split}}\h^{n}K(c_{\bullet})&&\text{Corollary~\ref{dg-hty-classes}}\\
    &\cong\colim_{c_{\bullet}\to c}\h^{n}K(c_{\bullet})&&\text{assumption}
  \end{align*}
\end{proof}

\begin{rem}
  The hypothesis of the Proposition, \ie{} that every hypercover
  admits a split refinement, is satisfied in many cases, \eg{} when
  $({\cal C},\tau)$ is a Verdier site,
  see~\cite[Thm.~8.6]{dugger-hollander-isaksen}.

  Moreover, in these cases the proposition represents another approach
  to Theorem~\ref{local-model-structure}. Indeed, the essential point,
  as we mentioned in Remark~\ref{existence-local-model-structure}, is
  the description of the $\tau$-fibrant objects in $\U{\cal
    C}/\tau$. Since $\Lambda(c_{\bullet})\to\Lambda(c)$ is a
  $\tau$-local equivalence for each $\tau$-hypercover $c_{\bullet}\to
  c$ (Fact~\ref{hypercover-acyclic}) it is clear that $\tau$-fibrant
  objects satisfy $\tau$-descent. Conversely, suppose $K\in\U{\cal C}$
  satisfies $\tau$-descent and choose a $\tau$-fibrant replacement
  $f:K\to K'$. Using the previous proposition we will prove that $f$
  is a quasi-isomorphism.

  Fix $c\in{\cal C}$ and $n\in\Z$. Consider the following commutative
  diagram:
  \begin{equation*}
    \xymatrix{\colim_{c_{\bullet}\to
        c}\h^{n}K(c_{\bullet})\ar[r]^-{\sim}&\hom_{\Ho(\U{\cal
          C}/\tau)}(\Lambda(c),K[-n])\ar[d]^{f}\\
    \h^{n}K(c)\ar[u]\ar[r]_{f}&\h^{n}K'(c)}
  \end{equation*}
  The left vertical arrow is an isomorphism since $K$ satisfies
  $\tau$-descent. The right vertical arrow is an isomorphism since
  $K'$ is $\tau$-fibrant. Thus the claim.
\end{rem}

\subsection{Complements}
In this last paragraph we discuss two further aspects of the local dg
homotopy theory: monoidal structures, and closure of fibrant objects
under certain operations.

\begin{pro}
  Assume that either of the following conditions is satisfied:
  \begin{enumerate}
  \item ${\cal C}$ is cartesian monoidal.
  \item For any objects $c,d\in{\cal C}$,
    $\Lambda(c)\otimes\Lambda(d)$ is projective, and $({\cal C},\tau)$
    has enough points.
  \end{enumerate}
  Then $\U{\cal C}/\tau$ is a symmetric monoidal model category for
  the objectwise tensor product.
\end{pro}
\begin{proof}\mbox{}
  \begin{enumerate}
  \item If ${\cal C}$ is cartesian monoidal, we may adapt the proof
    of~\cite[Thm.~4.58]{barwick-modcat-bousfield}. By~\cite[Pro.~4.47]{barwick-modcat-bousfield},
    it suffices to prove that for each $d\in{\cal C}$, and each
    $\tau$-local $K\in\U{\cal C}$, the internal hom object
    $[\Lambda(d),K]$ is $\tau$-local. Thus let $c_{\bullet}\to c$ be a
    $\tau$-hypercover. Using the commutative diagram
    \begin{equation*}
      \xymatrix{\dghom(\Lambda(c),[\Lambda(d),K])\ar[r]\ar[d]_{\sim}&\dghom(\Lambda(c_{\bullet}),[\Lambda(d),K)]\ar[d]^{\sim}\\
        \dghom(\Lambda(c)\otimes\Lambda(d),K)\ar[d]_{\sim}\ar[r]&\dghom(\Lambda(c_{\bullet})\otimes\Lambda(d),K)\ar[d]^{\sim}\\
       \dghom(\Lambda(c\times
       d),K)\ar[r]&\dghom(\Lambda(c_{\bullet}\times d),K)}
    \end{equation*}
    we reduce to showing that $c_{\bullet}\times d\to c\times d$ is a
    $\tau$-local equivalence of simplicial presheaves. This follows
    from the fact that homotopy groups and sheafification commute with
    finite products.
  \item By Lemma~\ref{uni-dg-mon-mod}, $\U{\cal C}$ is a symmetric
    monoidal model category. The result then follows from the proof
    of~\cite[Pro.~4.4.63]{ayoub07-thesis} (whereas the statement in
    loc.\,cit.\ misses the first hypothesis above).\qedhere{}
  \end{enumerate}
\end{proof}

Our description of $\tau$-fibrant objects in
Theorem~\ref{local-model-structure} allows one to prove easily that
these are closed under various operations. In the following lemmas we
discuss two examples.

\begin{lem}\label{bounded-tot-fibrant}
  Let $K_{\bullet}$ be a bounded complex of $\tau$-fibrant objects in
  $\U{\cal C}$. Then $\tots K_{\bullet}\in\U{\cal C}$ is
  $\tau$-fibrant.
\end{lem}
\begin{proof}
  Let $c_{\bullet}\to c$ be a $\tau$-hypercover. We know that for any
  $l\in\Z$, $K_{l}(c)\to K_{l}(c_{\bullet})$ is a
  quasi-isomorphism. Since $K_{\bullet}$ is bounded below, it follows
  from Lemma~\ref{lem:total-complex-qis} that also
  \begin{equation*}
    \tots (K_{\bullet}(c))\to \tots (K_{\bullet}(c_{\bullet}))
  \end{equation*}
  is a quasi-isomorphism. Since $K_{\bullet}$ is bounded (hence
  $\tots$ and $\totp$ agree), one easily checks that this morphism can
  be identified with
  \begin{equation*}
    (\tots K_{\bullet})(c)\to (\tots K_{\bullet})(c_{\bullet}).\qedhere{}
  \end{equation*}
\end{proof}

Let $\kappa$ be a regular cardinal. We say that the site $({\cal
  C},\tau)$ is \emph{$\kappa$-noetherian} if every cover $\{c_{i}\to
c\}_{i\in I}$ has a subcover $\{c_{i}\to c\}_{i\in J\subset I}$ with
$|J|<\kappa$. An $\aleph_{0}$-noetherian site is called simply
\emph{noetherian}, as in~\cite[§III.3]{milne:etale-coh}. Also, recall
the notion of Verdier sites
from~\cite[Def.~8.1]{dugger-hollander-isaksen}.

\begin{lem}
  Let $({\cal C},\tau)$ be a $\kappa$-noetherian Verdier site,
  $\kappa>\aleph_{0}$. Then $\tau$-fibrant objects in $\U{\cal C}$ are
  closed under $\kappa$-filtered colimits.
\end{lem}
\begin{proof}
  By~\cite[Rem.~8.7]{dugger-hollander-isaksen}, there is a dense set
  of $\tau$-hypercovers ${\cal S}$ such that for each $c_{\bullet}\to
  c\in{\cal S}$ and each $n\in\N$, $c_{n}$ is a coproduct
  $c_{n}\cong\coprod_{i\in I_{n}}c_{n,i}$ with $c_{n,i}$ representable
  and $|I_{n}|<\kappa$. By Theorem~\ref{hypercover-localization} and
  Lemma~\ref{local=descent}, being $\tau$-fibrant is equivalent to
  satisfying ${\cal S}$-descent. Now let $K:J\to\U{\cal C}$ be a
  $\kappa$-filtered diagram of $\tau$-fibrant objects, and
  $c_{\bullet}\to c\in{\cal S}$. The claim then follows from the
  isomorphism
  \begin{align*}
    (\colim_{j}K(j))(c_{\bullet})&\cong\totp
    (\colim_{j}K(j)_{p}(c_{q}))_{p,q}\\
    &\cong \totp (\prod_{i\in
      I_{q}}\colim_{j}K(j)_{p}(c_{q,i}))_{p,q}\\
    &\cong \colim_{j}\totp (\prod_{i\in
      I_{q}}K(j)_{p}(c_{q,i}))_{p,q}\\
    &\cong \colim_{j}(K_{j}(c_{\bullet})),
  \end{align*}
  as $\kappa$-filtered colimits commute with products indexed by
  cardinals smaller than $\kappa$.
\end{proof}

\begin{lem}
  Let $({\cal C},\tau)$ be a noetherian Verdier site. Any filtered
  colimit of bounded above $\tau$-fibrant objects in $\U{\cal C}$ is
  $\tau$-fibrant.
\end{lem}
\begin{proof}
  The proof is essentially the same as in the previous lemma. We must
  assume bounded above objects so that the product totalization
  involves only finitely many factors in each degree hence commutes
  with filtered colimits.
\end{proof}

\section{Fibrant replacement}
\label{sec:fib}
In this section we would like to give an ``explicit'' fibrant
replacement functor in $\U{\cal C}/\tau$ using the Godement
resolution. It is a direct translation of the analogous construction
for simplicial (pre)sheaves in~\cite[p.~66ff]{morel-voevodsky-a1},
with, again, the only problem created by the unboundedness of our
complexes. We first establish the tools to overcome this difficulty.

\subsection{Local model structure and truncation}
Let $n\in\Z$ and consider the functor $\Gamma\tau_{\geq n}:\U{\cal
  C}\to \simp^{\mathrm{op}}\psh({\cal C})$. Applying it objectwise,
this generalizes to a functor defined on diagrams with values in
$\U{\cal C}$ which we still denote by $\Gamma\tau_{\geq n}$.
\begin{lem}\label{dold-kan-holim}
  The canonical arrow
  \begin{equation*}
    \dR\lim_{\simp}\Gamma\tau_{\geq n}K\to\Gamma\tau_{\geq n}\dR\lim_{\simp}K
  \end{equation*}
  is a weak homotopy equivalence for every $K\in(\U{\cal C})^{\simp}$.
\end{lem}
\begin{proof}
  One way to see this is as follows. $\Gamma\tau_{\geq n}$ is a right
  Quillen functor for the projective model structures on ${\cal
    M}:=\U{\cal C}$ and ${\cal N}:=\psh({\cal
    C},\sset)$. It follows that the induced morphism of
  derivators $\mathbb{D}_{{\cal M}}\to\mathbb{D}_{{\cal
      N}}$ 
  is continuous (see~\cite[Pro.~6.12]{cisinski-imagesdirectes}), in
  particular it commutes with homotopy limits. The claim now follows
  from the fact that $\Gamma\tau_{\geq n}$ takes quasi-isomorphisms to
  weak homotopy equivalences hence doesn't need to be derived.
\end{proof}

\begin{pro} \label{truncation}\mbox{}
  \begin{enumerate}
  \item For a morphism $f : K \to K'$ in $\U{\cal C}$ the following
    are equivalent:
    \begin{enumerate}
    \item $f$ is a $\tau$-local equivalence.
    \item $\Gamma\tau_{\geq n}f$ is a $\tau$-local equivalence for all $n\in\Z$.
    \item $\Gamma\tau_{\geq n}f$ is a $\tau$-local equivalence for $n\ll 0$.
    \end{enumerate}
  \item For $K \in \U{\cal C}$ the following are equivalent:
    \begin{enumerate}[label=(\alph{enumii}),ref=(\alph{enumii})]
    \item $K$ is $\tau$-fibrant.\label{truncation.2.1}
    \item $\Gamma\tau_{\geq n}K$ is $\tau$-fibrant for all $n\in\Z$.\label{truncation.2.2}
    \item $\Gamma\tau_{\geq n}K$ is $\tau$-fibrant for $n\ll 0$.\label{truncation.2.3}
    \end{enumerate}
 \end{enumerate}
\end{pro}
\begin{proof}\mbox{}
  \begin{enumerate}
  \item This is obvious since $\tau$-local equivalences are defined
    via (the sheafification of) the homology groups which coincide
    with the homotopy groups after applying $\Gamma$.
  \item The implication
    ``\ref{truncation.2.1}$\Rightarrow$\ref{truncation.2.2}'' follows
    from Lemma~\ref{dold-kan-local}. The implication
    ``\ref{truncation.2.2}$\Rightarrow$\ref{truncation.2.3}'' is
    trivial. For the implication
    ``\ref{truncation.2.3}$\Rightarrow$\ref{truncation.2.1}'' let
    $f:K\to K'$ be a $\tau$-fibrant replacement. Again by
    Lemma~\ref{dold-kan-local}, $\Gamma\tau_{\geq n}(f)$ is a
    $\tau$-local equivalence between $\tau$-fibrant objects hence it
    is a sectionwise weak equivalence. It follows that $\tau_{\geq
      n}(f)$ is a sectionwise weak equivalence. As $f$ is the filtered
    colimit of $\tau_{\geq n}(f)$, $f$ is a sectionwise weak
    equivalence.\qedhere{}
  \end{enumerate}
\end{proof}

\subsection{Godement resolution}

Now suppose that $({\cal C},\tau)$ has enough points. This means that
there is a set ${\cal P}$ of morphisms of sites $p:\set\to ({\cal
  C},\tau)$ such that a morphism $f$ of sheaves of sets on ${\cal C}$
is an isomorphism if and only if $p^*f$ is an isomorphism for all
$p\in {\cal P}$. There is an induced morphism of sites $\set^{{\cal
    P}}\to ({\cal C},\tau)$, and we denote by $(a^{*},a_{*}):\U{\cal
  C}\to \ch(\Lambda)^{{\cal P}}$ the induced adjunction. The
associated comonad induces functorially for each $K\in\U{\cal C}$ a
coaugmented cosimplicial object $K\to G^{\bullet}(K)$, where
$G^{n}(K)=(a_{*}a^{*})^{n+1}(K)\in\U{\cal C}$. The \emph{Godement
  resolution} of $K$ is defined to be
\begin{equation*}
  \god(K):=\totp(G^{\bullet}(K))
\end{equation*}
which according to Lemma~\ref{holim-complexes} is a model for
$\dR\lim_{\simp}G^{\bullet}(K)$.

Recall~\cite[Def.~1.31]{morel-voevodsky-a1} that the site $({\cal
  C},\tau)$ is said to be of finite type if ``Postnikov towers
converge''.
\begin{thm}
  There is a functor $\god:\U{\cal C}\to \U{\cal C}$ and a natural
  transformation $\id\to\god$ satisfying:
\begin{enumerate}
\item $\god$ is an exact functor of abelian categories.\label{god.exact}
\item $\god$ takes each presheaf of complexes to a $\tau$-fibrant sheaf
  of complexes.\label{god.fibrant}
\item $\god$ takes fibrations (\ie{} degreewise surjections) to
  $\tau$-fibrations.
\item If $({\cal C},\tau)$ is a finite type site, then $K \to \god(K)$
  is a $\tau$-local equivalence for any $K$.
\end{enumerate}
\end{thm}
\begin{proof}\mbox{}
  \begin{enumerate}
  \item $\god$ is the composition of exact functors thus exact.
  \item We use Proposition~\ref{truncation} to check that $\god(K)$ is
    $\tau$-fibrant. Thus let $n\in\Z$, and $c_{\bullet}\to c$ a
    $\tau$-hypercover. We need to check that
    \begin{equation*}
      \Gamma\tau_{\geq n}\god(K)(c)\to\dR\lim_{\simp}\Gamma\tau_{\geq n}\god(K)(c_{\bullet})
    \end{equation*}
    is a weak homotopy equivalence. This will follow
    from~\cite[Pro.~1.59]{morel-voevodsky-a1} if we can prove
    that the canonical arrow
    \begin{equation*}
      \god(\Gamma\tau_{\geq n}K)\to\Gamma\tau_{\geq n}\god(K)
    \end{equation*}
    is an objectwise weak homotopy equivalence, where the left hand
    side denotes the Godement resolution for simplicial (pre)sheaves
    as defined in~\cite[p.~66]{morel-voevodsky-a1}, analogous to our
    construction above. By Lemma~\ref{dold-kan-holim}, we see that
    $\Gamma\tau_{\geq n}$ commutes with $\dR\lim_{\simp}$ up to
    objectwise weak equivalence, so we reduce to show that it also
    commutes with $a_{*}a^{*}$ up to objectwise weak
    equivalence. 

    $a_{*}a^{*}$ is applied degreewise and is a composition of
    left-exact functors hence clearly commutes with $\tau_{\geq
      n}$. It is also clear that $a_{*}a^{*}$ commutes with the Moore
    complex functor therefore the same holds for the quasi-inverse
    $\Gamma$. Finally, $a_{*}a^{*}$ commutes with the forgetful
    functor $\Mod{\Lambda}\to\set$.
  \item Let $f$ be an epimorphism with kernel $K$ in $\U{\cal C}$. By
    part~\ref{god.exact}, $\god(f)$ is an epimorphism with kernel
    $\god(K)$, which is $\tau$-fibrant by
    part~\ref{god.fibrant}. $\god(f)$ is thus a $\tau$-fibration by
    Theorem~\ref{local-model-structure}.
  \item Again, by Proposition~\ref{truncation}, we need to check that
    \begin{equation*}
      \Gamma\tau_{\geq n}K\to\Gamma\tau_{\geq n}\god(K)
    \end{equation*}
    is a $\tau$-local equivalence for all $n\in\Z$. But by the same
    reasoning as in part~\ref{god.fibrant}, the target of this
    morphism is identified (up to sectionwise weak homotopy
    equivalence) with $\god(\Gamma\tau_{\geq n}K)$ hence the claim
    follows from~\cite[Pro.~1.65]{morel-voevodsky-a1}.
  \end{enumerate}
\end{proof}
\bibliographystyle{plain} 

\end{document}